\documentclass[12pt]{amsart}
\usepackage{amssymb,amsmath, amsthm,latexsym,verbatim}
\usepackage{a4wide}
\usepackage{hyperref}

\theoremstyle{plain}
\newtheorem{theorem}{Theorem}
\newtheorem{lem}{Lemma}[section]

\newtheorem{prop}[lem]{Proposition}
\newtheorem{corollary}[lem]{Corollary}

\newtheorem{theostar}{Theorem}

\theoremstyle{definition}

\theoremstyle{remark}
\newtheorem {bmrk}[theorem] {Remark}

\newcommand{\C}{\mathbb{C}}
\newcommand{\R}{\mathbb{R}}
\newcommand{\Z}{\mathbb{Z}}
\newcommand{\Q}{\mathbb{Q}}

\newcommand{\suL}{\mathfrak{su}}

\newcommand{\diag}{\operatorname{diag}}
\newcommand{\rk}{\operatorname{rk}}
\newcommand{\gL}{\mathfrak{g}}
\newcommand{\kL}{\mathfrak{k}}

\newcommand{\aL}{\mathfrak{a}}
\newcommand{\pL}{\mathfrak{p}}
\newcommand{\nL}{\mathfrak{n}}

\newcommand{\Real}{\operatorname{Re}}
\newcommand{\Symm}{\operatorname{Symm}}

\newcommand{\SL}{\operatorname{SL}}
\newcommand{\SU}{\operatorname{SU}}
\newcommand{\Spin}{\operatorname{Spin}}

\newcommand{\SO}{\operatorname{SO}}
\newcommand{\Tr}{\operatorname{Tr}}
\newcommand{\pr}{\operatorname{pr}}

\newcommand{\Id}{\operatorname{Id}}
\newcommand{\Ad}{\operatorname{Ad}}
\newcommand{\vol}{\operatorname{vol}}
\newcommand{\Gal}{\operatorname{Gal}}
\newcommand{\End}{\operatorname{End}}

\newcommand{\ad}{\operatorname{ad}}

\newcommand{\Rep}{\operatorname{Rep}}
\newcommand{\Hom}{\operatorname{Hom}}
\newcommand{\Aut}{\operatorname{Aut}}

\newcommand{\reg}{\operatorname{reg}}

\newcommand{\ovl}[1]{\overline #1}

\newcommand{\bs}{\backslash}

\newcommand{\id}{\operatorname{Id}}

\newcommand{\G}{\mathrm{G}}

\newcommand{\so}{\mathcal{O}}

\newcommand{\fra}{\mathfrak{A}}
\newcommand{\frb}{\mathfrak{B}}

\newcommand{\fri}{\mathfrak{I}}

\newcommand{\frf}{\mathfrak{f}}

\newcommand{\tors}{\mathrm{tors}}

\newcommand{\CC}{\mathbb C}

\newcommand{\ZZ}{\mathbb Z}
\newcommand{\HH}{\mathbb H}

\newcommand{\QQ}{\mathbb Q}

\newcommand{\Ade}{\mathbb A}

\title[Torsion in symmetric powers]{The torsion in symmetric powers on congruence subgroups of Bianchi groups}

\author{Jonathan Pfaff and Jean Raimbault}
\address{Universit\"at Bonn\\
Mathematisches Institut\\
Endenicher Alle 60\\
D -- 53115 Bonn, Germany}
\email{pfaff@math.uni-bonn.de}
  \address{Institut de Math\'ematiques de Toulouse ; UMR5219 \\ Universit\'e de Toulouse ; CNRS \\ UPS IMT, F-31062 Toulouse Cedex 9, France}
  \email{Jean.Raimbault@math.univ-toulouse.fr}
\begin{document}

\maketitle

\begin{abstract}
In this paper we prove that for a fixed neat principal congruence subgroup of a Bianchi group the order of the torsion part of its second cohomology group with coefficients in an integral lattice associated to the $m$-th symmetric power of the standard representation of $\SL_2(\C)$ grows exponentially in $m^2$. We give upper and lower bounds for the growth rate. Our result extends a a result of W. M\"uller and S. Marshall, who proved the corresponding statement for closed arithmetic 3-manifolds, to the finite-volume case.  We also prove a limit multiplicity formula for twisted combinatorial Reidemeister torsion on higher dimensional hyperbolic manifolds. 
\end{abstract}

\setcounter{tocdepth}{1}
\tableofcontents
\setcounter{tocdepth}{2}

\section{Introduction}

\setcounter{equation}{0}
The torsion in the cohomology of arithmetic groups has recently attracted new interest from number theorists. Without aiming at completeness, we refer for example to \cite{BV}, \cite{CV}, \cite{Eme} and \cite{Scholze}. In this paper, we study the twisted cohomological torsion quantitatively for a fixed principal congruence subgroup of a Bianchi group under a variation of the local system. Bianchi groups represent all classes of non-uniform lattices in $\SL_2(\CC)$ ; thus our result complements the study of this question for arithmetic lattices in $\SL_2(\CC)$ defined over imaginary quadratic fields done by Simon Marshall and Werner M\"uller in \cite{MaM} (where the authors give an equality for the asymptotic torsion size, while we only get upper and lower bounds for the growth rate). 

To state our main result more precisely, we need to introduce some notation. Let $D\in\mathbb{N}$ be square-free and let $F=\Q(\sqrt{-D})$ be the associated imaginary quadratic number field with ring of integers $\mathcal{O}_D$. Let $\Gamma_D:=\SL_2(\mathcal{O}_D)$. Then $\Gamma_D$ is an arithmetic subgroup of $\SL_2(\C)$ which acts on $\SL_2(\C)/\SU(2)\cong\mathbb{H}^3$ and the quotient $\Gamma_D\backslash\mathbb{H}^3$ is a hyperbolic orbifold of finite volume. If $\aL$ is a non-zero ideal of $\mathcal{O}_D$, we let $\Gamma(\aL)$ denote the principal congruence subgroup of level $\aL$. This is a finite-index subroup of $\Gamma_D$ which is neat (i.e. none of its non-unipotent elements have a root of unity as an eigenvalue, in particular it is torsion-free) as soon as the norm $N(\aL)$ is sufficiently large ($N(\aL)\ge 9$ suffices). We shall assume this from now on. Thus, $X_{\aL}:=\Gamma(\aL)\backslash\mathbb{H}^3$ is an arithmetic hyperbolic manifold of finite volume. It is never compact and has finitely many cusps, whose number we shall denote by $\kappa(\Gamma(\aL))$. For $m\in\mathbb{N}$ let $\rho_m$ be the natural representation of $\SL_2(\C)$ on the $m$th symmetric power $V(m):=\Symm^m\C^2$ of it's standard representation on $\CC^2$. Then there exists a $\mathbb{Z}$-lattice $\Lambda(m)$ in $V(m)$ which is preserved by $\rho_m(\Gamma_D)$ (one can simply take $\Lambda(m) = \Symm^m\so_D^2$). 

Now one can form the integral cohomology groups $H^*(\Gamma(\aL);\Lambda(m))$ of the $\Gamma(\aL)$-modules $\Lambda(m)$. These are finitely generated abelian groups and thus they split as 
\[
H^*(\Gamma(\aL);\Lambda(m))=H^*(\Gamma(\aL);\Lambda(m))_{free}\oplus H^*(\Gamma(\aL);\Lambda(m))_{tors},
\] 
where the first group in this decomposition is a free finite-rank $\Z$-module and the second group is a finite abelian group. Moreover, $H^*(\Gamma(\aL);\Lambda(m))_{free}$ is a lattice in the real cohomology $H^1(\Gamma(\aL),V(m))$. In this paper we are interested in the behaviour of the size of the cohomology group $H^*(\Gamma(\aL);\Lambda(m)))$ as $m$ goes to infinity. First we note that
\begin{equation}\label{difree}
\dim H^1(\Gamma(\aL),V(m)) = \dim H^2(\Gamma(\aL),V(m)) = \kappa(\Gamma(\aL)
\end{equation}
for each $m\in\mathbb{N}$. This is easy to verify, see for example section \ref{comp_tors}. On the other hand, we will show that in degree 2 the size of the torsion part grows exponentially in $m^2$ as $m\to \infty$ and we will specify the growth rate.  More precisely, the main result of this paper is the following theorem (we stated it for the lattices $\Lambda(m)$ but we prove that the growth rates are the same for any sequence of $\Gamma(\aL)$-invariant lattices in $V(m)$---see Proposition \ref{comp_homtors}). 

\begin{theostar}\label{Mainthrm}
There exist constants $C_1(\Gamma_D)>0$, $C_2(\Gamma_D)>0$, which depend only on $\Gamma_D$ such that for each non-zero ideal $\aL$ of $\mathcal{O}_D$ with $N(\aL)>C_1(\Gamma_D)$ one has 
\begin{align}\label{esttorsb}
\liminf_{m\to \infty}\frac{\log|H^2(\Gamma(\aL);\Lambda(m))_{tors}|}{m^2}\geq \frac 1 2 \cdot\frac{\vol(X_{\aL})}{\pi}\left(1-\frac{C_1(\Gamma_D)}{N(\aL)}\right)>0
\end{align}
and 
\begin{align}\label{esttorsab}
\limsup_{m\to\infty}\frac{\log|H^2(\Gamma(\aL);\Lambda(m))_{tors}|}{m^2}\leq \frac{\vol(X_{\aL})}{\pi}\left(1+\frac{C_2(\Gamma_D)}{N(\aL)}\right).
\end{align}
Finally, we also have
\begin{align}\label{esttorsh1}
|H^1(\Gamma(\aL);\Lambda(m))_{tors}|=O(m\log m),
\end{align}
as $m\to\infty$. 
\end{theostar}

If the class number of $F$ is one and $\mathcal{O}_D^*=\{\pm 1\}$, we can take $C_1(\Gamma_D)=4$ (this is valid only for $D = -7,-11,-19,-43,-67$ and $-163$ by the Stark--Heegner Theorem). We shall now briefly sketch our method to prove our main result. The main point in our case is to establish the lower bound on the torsion given in \eqref{esttorsb}, i.e. to establish its exponential growth in $m^2$. Let us point out that there are two severe difficulties in the present non-compact case which are not present in the case of compact arithmetic 3-manifolds mentioned above. Firstly, the use of analytic torsion as a main tool is more complicated. Secondly, the real cohomology $H^*(\Gamma,V(m))$ does not vanish in our situation. We shall now describe these issues in more detail. 
 
As already observed by Nicolas Bergeron and Akshay Venkatesh in \cite{BV}, the size of cohomological torsion is closely related to the Reidemeister torsion of the underlying manifold with coefficients in the underlying local system. In the present finite volume case, one has to work with the twisted Reidemeister torsion of the Borel--Serre compactification $\overline{X}$ of $X$. For technical reasons, in most of the paper we also symmetrize the lattice $\Lambda(m)$ to a lattice $\bar{\Lambda}(m)$ in $\bar V(m):=V(m)\oplus V(m)^*$ which it self-dual over $\Z$ (it is then not hard to deduce the estimates in Theorem \ref{Mainthrm} from their analogues for $\bar\Lambda(m)$-coefficients). The Reidemeister torsion is then defined with respect to a canonical basis in the cohomology $H^*(X;\bar{V}(m))$ using Eisenstein cohomology classes following G\"unter Harder \cite{Ha}. By a gluing formula for the Reidemeister  torsion, which was obtained by the first author in \cite{Pf} building on work of Matthias Lesch \cite{Lesch}, the Reidemeister torsion can be compared to the regularized analytic torsion of the manifold $X$. The asymptotic behaviour of the regularized analytic torsion in the finite volume case for a variation of the local system has already been studied by M\"uller and the first author in \cite{MPtors1} using the Selberg trace formula. Thus we have to study the error term which occured in \cite{Pf} in the comparison formula between analytic and Reidemeister torsion. It turns out that our study of the error term can be performed without any changes also in the higher dimensional situations. While we do not compute the error term explicitly, we bring it in a form which is sufficient for the application to cohomological torsion.The main point is that the error term depends only on the geometry of the cusps which is very restricted on such manifolds.  Also, along the line we can establish limit multiplicity formulae for twisted Reidemeister torsion in the spirit of \cite{BV}  on arithmetic hyperbolic manifolds of finite volume of arbitrary dimension, see Corollary \ref{Corollarylimmult}.

Now we turn to the second aforementioned difficulty. Since the real cohomology of $\Gamma(\aL)$ with coefficients in $\bar{V}(m)$ does not vanish in the finite-volume case, we also have to study certain volume factors occuring in the comparison formula between Reidemeister torsion and the size of cohomological torsion. Since our basis in the real cohomology is given by Eisenstein series, this leads to the question about the integrality of certain quotients of $L$-functions evaluated at positive integers. In the 3-dimensional case, these are Hecke L-functions and we can use the work of R. Damerell \cite{Damerell1},\cite{Damerell2} to control these quotients. At the moment, we do not know how to do this in higher dimensions. 

Concerning the other statements of our main theorem, at least with a worse constant, the upper bound \eqref{esttorsab} can be established in an elementary and completely combinatorial way, without referring to analytic or Reidemeister torsion, see Proposition \ref{propestab}. In fact, our approach for the upper bound on the combinatorial torsion generalizes easily to arbitrary dimensions and to arbitrary rays in the weight lattice; it is similar to that used by V. Emery \cite{Eme} or R. Sauer \cite{Sa}. We note that if we were able to obtain a proper limit for the Redemeister torsion instead of the quotient of two such, we would obtain an optimal upper bound for the exponential growth rate of the torsion in the second cohomology, by an argument similar to that used in the proof of the easy part of \cite[Lemma 6.14]{congruence}. The last estimate \eqref{esttorsh1} in our theorem is the easiest one to prove, and does not require analytic torsion or any sophisticated tool. Since we work with a $\QQ$-split group the representations are easier to analyze that in the nonsplit case which was dealt with in \cite[section 4]{MaM}---we can work globally from the beginning. 

We finally remark that Theorem \ref{Mainthrm} also holds with the same proof for slightly more general rays of local systems. Namely, the finite dimensional irreducible representations of $\SL_2(\C)$ are parametrized as $\Symm^{n_1}\otimes\overline{\Symm}^{n_2}$, where $n_1,n_2\in\mathbb{N}$ and where $\overline{\Symm^{n_2}}$ is the complex conjugate of $\Symm^{n_2}$. Each such representation space carries a canonical $\Z$-lattice preserved by the action of $\Gamma_D$. If we fix $n_1$ and $n_2$ with $n_1\neq n_2$ and let $\rho_{m}(n_1,n_2)$ be the representation $\Symm^{mn_1}\otimes\overline{\Symm^{mn_2}}$, then the analog of Theorem \ref{Mainthrm} holds if we replace the factor $m^2$ by $m\dim \rho_{m}(n_1,n_2)$ which grows as $m^3$ if both $n_1$ and $n_2$ are not zero. However, we can by no means remove the assumption $n_1\neq n_2$. In other words, the ray $\rho_m(1,1)=\Symm^{m}\otimes \overline{\Symm^{m}}$, which is the ray carrying cuspidal cohomology, cannot be studied by our methods. For a fixed $\Gamma(\aL)$, we can only show that the size of cohomological torsion with coefficients in the canonical lattice associated to $\rho_m(1,1,)$ grows at most as $m\rk \rho_{m}(1,1)=O(m^3)$, but we can say nothing about the existence of torsion along this ray, i.e. we cannot establish any bound from below. The reason is that here $0$ belongs to the essential spectrum of the twisted Laplacian in the middle dimension. Therefore, essentially none of the results on analytic torsion we use in our proof is currently available for $\rho_m(1,1)$ and also the regulator would be more complicated. We remark that, as far as we know, even in the compact case no result for the growth of torsion along this particular ray has been obtained. For an investigation  of the dimension of the (cuspidal) cohomology along this ray, we refer to \cite{FGT}. 

\medskip

This paper is organized as follows : in section \ref{reg_analytic} we introduce the analytic torsion for cusped manifolds, then in section \ref{tors_cong} we study it further for congruence subgroups of Bianchi groups. We introduce the combinatorial (Reidemeister) torsion in section \ref{comp_tors} and recall the Cheeger--M\"uller equality proven in \cite{Pf} there. Then we study the intertwining operators in section \ref{intertwine_adele} and \ref{C_denominator}, first computing them using adeles and then bounding their denominators. The last sections contain the proof of the main theorem: in section \ref{lower_bound} we combine the results of the previous sections to prove \eqref{esttorsb}, and we prove \eqref{esttorsab} in section \ref{upper_bound}.

\medskip

{\bf Acknowledgement.} The first author was financially supported by the DFG-grant PF 826/1-1. He gratefully
acknowledges the hospitality of Stanford University in 2014 and 2015.


\section{The regularized analytic torsion for coverings}
\label{reg_analytic}
\setcounter{equation}{0}
In this section we shall review the definition 
of regularized traces and the regularized analytic torsion of hyperbolic manifolds 
$X$ of finite volume. These objects are defined in terms 
of a fixed choice of  truncation parameters on $X$ and there 
are two different ways to  perform such a trunctation 
which are relevant in the present paper.
Firstly, one can define 
a truncation of $X$ via a fixed
choice of $\Gamma$-cuspidal parabolic subgroups of $G$. Secondly, if 
$X$ is a finite covering of a hyperbolic orbifold $X_0$, then a choice 
of truncation parameters on $X_0$ gives a truncation 
on $X$ in terms of which one can define another 
regularized analytic torsion. We shall compute 
the difference between the associate regularized analytic torsions explicitly. 
For more details we 
refer to \cite{MPtors1}, \cite{MPtors}.

We denote by $\SO^0(d,1)$ the identity-component of the 
isometry group of 
the standard quadratic form of signature $(d,1)$ 
on $\R^{d+1}$. Let $\Spin(d,1)$ denote 
the universal covering of $\SO^0(d,1)$. 
We let either $G:=\SO^0(d,1)$ or 
$G:=\Spin(d,1)$. 
We assume that $d$ is odd and write $d=2n+1$. Let $K:=\SO(d)$, if 
$G=\SO^0(d,1)$ or $K:=\Spin(d)$, if $G=\Spin(d,1)$. 
We let $\gL$ be the Lie algebra of $G$. Let 
$\theta$ denote the standard Cartan 
involution of $\gL$ and let $\gL=\kL\oplus\pL$ denote 
the associated Cartan decomposition of $\gL$, where
$\kL$ is the Lie algebra of $K$. Let $B$ be the Killing form. 
Then 
\begin{align}\label{metr}
\left<X,Y\right>:=-\frac{1}{2(d-1)}B(X,\theta Y)
\end{align}
is an inner product on $\gL$.  
Moreover, the globally 
symmetric space $G/K$, 
equipped with the $G$-invariant metric induced by
the restriction of \eqref{metr} to $\pL$ is isometric 
to the $d$-dimensional real hyperbolic space $\mathbb{H}^d$.
Let $\Gamma\subset G$ be a discrete, torsion-free subgroup. 
Then $X:=\Gamma\backslash\mathbb{H}^d$, equipped 
with the push-down of the metric on $\mathbb{H}^d$, is a $d$-dimensional 
hyperbolic manifold.
We let $P$ be a fixed  
parabolic subgroup of $G$ with Langlands 
decomposition $P=M_{P}A_{P}N_{P}$ as in \cite{MPtors1}. 
Let $\aL$ denote the Lie algebra of $A_{P}$ and 
$\exp:\aL\to A_{P}$ the exponential 
map. Then we fix a  restricted root $e_1$ 
of $\aL$ in $\gL$, let $H_1\in\aL$ 
be such that $e_1(H_1)=1$ and define $\iota_{P}:(0,\infty)\to A_{P}$ 
by $\iota_{P}(t):=\exp(\log tH_1)$.
If $P_1$ is another parabolic subgroup 
of $G$, we fix $k_P\in K$ with 
$P_1=k_{P_1}Pk_{P_1}^{-1}$ and define $A_{P_1}:=k_{P_1}A_{P}k_{P_1}^{-1}$, $M_{P_1}:=k_{P_1}M_{P_1}k_{P_1}^{-1}$
$N_P:=k_{P_1}N_{P_1}k_{P_1}^{-1}$. Moreover, for $t\in (0,\infty)$ we 
define $\iota_{P_1}(t):=k_{P_1}\iota_{P}(t)k_{P_1}^{-1}\in A_{P_1}$. For 
$Y>0$ we let $A_{P_1}[Y]:=\iota_{P_1}([Y,\infty))$. 
  
A parabolic subgroup $P_1$ of $G$ is called $\Gamma$-cuspidal 
if $\Gamma\cap N_{P_1}$ is a lattice in $N_{P_1}$. 
From now on, we assume that $\vol(X)$ 
is finite and that $\Gamma$ is neat in the sense of Borel, i.e. 
that $\Gamma\cap P_1=\Gamma\cap N_{P_1}$ for 
each $\Gamma$-cuspdail $P_1$.
If $P_1$ is $\Gamma$-cuspidal, 
then for $Y>0$ we put
\[
F_{P_1;\Gamma}(Y):=(\Gamma\cap N_{P_1})\backslash N_{P_1}\times A_{P_1}[Y]\cong(\Gamma\cap N_{P_1})\backslash N_{P_1}\times [Y,\infty). 
\]
We equip $F_{P_1;\Gamma}(Y)$ with the metric $y^{-2}g_{N_{P_1}}+y^{-2}dy^2$ where $g_{N_{P_1}}$ is 
the push-down of the invariant metric on $N_{P_1}$ induced by the innerer 
product \eqref{metr} restricted to $\mathfrak{n}_{P_1}$.

Let $\Rep(G)$ denote the set of finite-dimensional irreducible 
representations of $G$. For $\rho\in \Rep(G)$ the 
associated vector space $V_\rho$ posesses a distinugished 
inner product $\left<\cdot,\cdot,\right>_\rho$ which is called admissible and which is unique 
up to scaling. We shall fix an admissible inner product
on each $V_\rho$. 
If $\rho\in\Rep(G)$, then the restriction 
of $\rho$ to $\Gamma$ induces a flat vector bundle $E_\rho:=\tilde{X}\times_{\rho|_{\Gamma}}V_\rho$. This bundle is canonically 
isomorphic to the locally homogeneous bundle
$E'_\rho:=\Gamma\backslash G\times_{\rho|_{K}}V_\rho$
induced by the restriction 
of $\rho$ to $K$. In particular, since $\rho|_{K}$ is a unitary 
representation on $(V_{\rho},\left<\cdot,\right>_{\rho})$, the inner 
product $\left<\cdot,\right>_{\rho}$ induces a smooth bundle metric on 
$E'_\rho$ and therefore on $E_\rho$.  
For $p=0,\dots,d$ let 
$\Delta_p(\rho)$ denote the flat Hodge Laplacian 
acting on the smooth $E_\rho$-valued $p$-forms of $X$. Since $X$ 
is complete, $\Delta_p(\rho)$ with 
domain the smooth, compactly supported $E_\rho$-valued $p$-forms is essentially selfadjoint 
and its $L^2$-closure shall be denoted by the same symbol. 
Let $e^{-t\Delta_p(\rho)}$ denote the heat semigroup of $\Delta_p(\rho)$ and 
let 
\[
K_{X}^{\rho,p}(t,x,y)\in C^{\infty}(X\times X;E_\rho\boxtimes E_\rho^*)
\]
be the integral kernel of $e^{-t\Delta_p(\rho)}$.

We let $\mathfrak{P}_{\Gamma}$ be a fixed set of 
$\Gamma$-cuspidal parabolic subgroups of $G$. Then $\mathfrak{P}_{\Gamma}$ 
is non-empty if and only if $X$ is 
non-compact. Moreover, $\kappa(\Gamma):=\#\mathfrak{P}_\Gamma$ equals the 
number of cusps of $X$ which from now on we 
assume to be nonzero. 
The choice of $\mathfrak{P}_{\Gamma}$ and of a fixed base-point in $\tilde{X}$ determine an exhaustion 
of $X$ by smooth compact manifolds $X_{\mathfrak{P}_\Gamma}(Y)$ with boundary, $Y>>0$. This exhaustion depends on the choice of $\mathfrak{P}_\Gamma$.
Then one 
can show that the integral of $K_{X}^{\rho,p}(t,x,x)$ over $X_{\mathfrak{P}_\Gamma}(Y)$ 
has an asymptotic expansion
\begin{align}\label{Asexp1}
\int_{X_{\mathfrak{P}_\Gamma}(Y)}\Tr K_{X}^{\rho,p}(t,x,x)dx=
\alpha_{-1}(t)\log Y+\alpha_0(t)+o(1),
\end{align}                                                                          
as $Y\to\infty$, \cite{MPtors1}[section 5]. Now 
one can define the regularized trace 
$\Tr_{\reg;X;\mathfrak{P}_\Gamma} e^{-t\Delta_p(\rho)}$ of $e^{-t\Delta_p(\rho)}$ 
with respect to the choice of $\mathfrak{P}_{\Gamma}$ by 
$\Tr_{\reg;X;\mathfrak{P}_\Gamma} e^{-t\Delta_p(\rho)}:=\alpha_0(t)$, 
where $\alpha_0(t)$ is the constant term in the asymptotic 
expansion in \eqref{Asexp1}.

From now on, we also assume that there is a hyperbolic 
orbifold $X_0:=\Gamma_0\backslash\mathbb{H}^d$ such 
that $X$ is a finite covering of $X_0$. Let 
$\pi: X\to X_0$ denote the covering map. Then if 
a set of truncation parameters on $X_0$, or in other words a 
set $\mathfrak{P}_{\Gamma_0}$ of representatives of $\Gamma_0$-cuspidal parabolic subgroups 
are fixed, one obtains truncation parameters on 
$X$ by pulling back the truncation on $X_0$ via $\pi$. 
One can again show that there is an asymptotic expansion
\begin{align}\label{Asexp}
\int_{\pi^{-1}X_{\mathfrak{P}_{\Gamma_0}}(Y)}\Tr K_{X}^{\rho,p}(t,x,x)dx=
\tilde{\alpha}_{-1}(t)\log Y+\tilde{\alpha}_0(t)+o(1),
\end{align}
as $Y\to\infty$ and one can define the regulraized trace 
with respect to the truncation parameters on $X_0$ as $\Tr_{\reg;X;X_0}e^{-t\Delta_p(\rho)}:=\tilde{\alpha}_0(t)$. This regulararized 
trace depends only on the choice of a set $\mathfrak{P}_{\Gamma_0}$ of 
representatives of $\Gamma_0$-cuspidal parabolic subgroups of $G$. Put
\begin{align}\label{defKX}
K_{X;\mathfrak{P}_\Gamma}(t,\rho):=\sum_p (-1)^pp\Tr_{\reg;\mathfrak{P}_\Gamma} e^{-t\Delta_p(\rho)};\: 
K_{X;X_0}(t,\rho):=\sum_p (-1)^pp\Tr_{\reg;X;X_0} e^{-t\Delta_p(\rho)}.
\end{align}
Now assume that $\rho$ satisfies $\rho\neq\rho_{\theta}$.
Then one defines the analytic torsion with respect 
to the two truncations of $X$ by  
\begin{align}
&\log T_{X;\mathfrak{P}_\Gamma}(\rho):=\frac{1}{2}\frac{d}{ds}\biggl|_{s=0}\left(\frac{1}{\Gamma(s)}\int_0^\infty t^{s-1}K_{X;\mathfrak{P}_\Gamma}(t,\rho)dt\right);\label{Tors1}\\
&\log T_{X;X_0}(\rho):=\frac{1}{2}\frac{d}{ds}\biggr|_{s=0}\left(\frac{1}{\Gamma(s)}\int_0^\infty t^{s-1}K_{X;X_0}(t,\rho)\right)\label{Tors2}.
\end{align}
Here the integrals
converge absolutely and locally uniformly for $\Re(s)>d/2$ and 
are defined near $s=0$ by analytic continuation, \cite[section 7]{MPtors1}, \cite[section 9]{MPtors}.

To compare the two analytic torsions in \eqref{Tors1} and \eqref{Tors2}, we need to 
introduce some more notation. 
We fix $\mathfrak{P}_{\Gamma_0}=\{P_{0,1},\dots,P_{0,\kappa(\Gamma_0)}\}$ and $\mathfrak{P}_{\Gamma}=\{P_1,\dots,P_{\kappa(\Gamma)}\}$. 
Then for each $P_j\in\mathfrak{P}_{\Gamma}$ there exists a unique $l(j)\in\{1,\dots,\kappa(\Gamma_0)\}$ 
and a $\gamma_j\in\Gamma_0$ such that $\gamma_jP_j\gamma_j^{-1}=P_{0,l(j)}$. 
Write 
\begin{align}\label{gammprime}
\gamma_j=n_{0,l(j)}\iota_{P_{0,l(j)}}(t_{P_j})k_{0,l(j)},  
\end{align}
$n_{0,l(j)}\in N_{P_{0,l(j)}}$, $t_{P_j}\in(0,\infty)$,
$\iota_{P_{0,l(j)}}(t_{P_j})\in A_{P_{0,l(j)}}$ as above, and $k_{0,l(j)}\in K$.
Since 
$P_{0,l(j)}$ equals its normalizer,  
the projection of $\gamma_j$ to $\bigl(\Gamma_0\cap P_{0,l(j)}\bigr)\backslash \Gamma_0$ 
is unique. Moreover, since
$P_{0,l(j)}$ is $\Gamma_0$-cuspidal, one has 
$\Gamma_0\cap P_{0,l(j)}=\Gamma_0\cap N_{P_{0,l(j)}}M_{P_{0,l(j)}}$. Thus 
$t_{P_j}$ depends only on $\mathfrak{P}_{\Gamma_0}$ and $P_j$.

Now the analytic torsions $\log T_{\mathfrak{P}_{\Gamma}}(X_1;E_\rho)$ 
and $\log T_{X_0}(X;E_\rho)$ are compared in  
the following proposition.
\begin{prop}\label{Compartors}
One has
\[
\log T_{\mathfrak{P}_{\Gamma}}(X;E_\rho)=\log T_{X_0}(X;E_\rho)+\sum_{P_j\in\mathfrak{P}_{\Gamma}}\log(t_{P_j})\left(\sum_{k=0}^n\frac{(-1)^k\dim(\sigma_{\rho,k})\lambda_{\rho,k}}{2}\right).
\]
\end{prop}
\begin{proof}
This follows by an application of a theorem of Kostant \cite{Kostant} on 
nilpotent Lie algebra cohomology.
For $p\in\{0,\dots,d\}$ define a representation 
of $K$ on $\Lambda^p\pL^*\otimes V(\rho)$ by $\nu_p(\rho):=\Lambda^p\Ad^*\otimes\rho$. 
Let $\tilde{E}_{\nu_p(\rho)}:=G\times_{\nu_p(\rho)}K$, which is a homogeneous 
vector bundle over $\mathbb{H}^d=G/K$. 
Let $\Omega$ be the Casimir element of $\gL$. Then $-\Omega+\rho(\Omega)$ 
induces canonically a Laplace-type operator $\tilde{\Delta}_p(\rho)$ which 
acts on the smooth sections of $\tilde{E}_{\nu_p(\rho)}$. The heat 
semigroup of $e^{-t\tilde{\Delta}_p(\rho)}$ is canonically represented by 
a smooth function $H_t^{p,\rho}:G\to\End(\Lambda^p\pL^*\otimes V_\rho)$, 
\cite[section 4, section 7]{MPtors1}. Let 
$h_t^{p,\rho}:=\Tr H_t^{p,\rho}$ and put 
\[
k_t^{p,\rho}:=\sum_{p}(-1)^pph_t^{p,\rho}.
\]
Then by the definition of the regularized traces, one has 
\begin{align}\label{Hilfsgl1}
\int_{X_{\mathfrak{P}_{\Gamma}(Y)}}\left(\sum_{\gamma\in\Gamma} k_t^{p,\rho}(x^{-1}\gamma x)\right)dx
=\alpha_{-1}(t;\rho)\log{Y}+K_{X;\mathfrak{P}_{\Gamma}}(t,\rho)+o(1),
\end{align}
as $Y\to\infty$, resp. 
\begin{align}\label{Hilfsgl2}
\int_{\pi^{-1}(X_{\mathfrak{P}_{\Gamma_0}(Y)})}\left(\sum_{\gamma\in\Gamma} k_t^{p,\rho}(x^{-1}\gamma x)\right)dx
=\tilde{\alpha}_{-1}(t;\rho)\log{Y}+
K_{X;X_0}(t,\rho)+o(1),
\end{align}
as $Y\to\infty$, where we use the notation \eqref{defKX}. 
On the other hand, for $k=0,\dots,n$ let $h_t^{\sigma_{\rho,k}}\in C^\infty(G)$ 
be defined as in \cite[(8.7)]{MPtors1}. 
If we apply the same considerations as in \cite[section 6]{MPtors} to the functions $h_t^{\sigma_{\rho,k}}$, then 
combining \cite[Proposition 8.2]{MPtors1}, \eqref{Hilfsgl1} and \eqref{Hilfsgl2} we obtain
\begin{align*}
K_{X;X_0}(t,\rho)=K_{X;\mathfrak{P}_{\Gamma}}(t,\rho)-
\sum_{P_j\in\mathfrak{P}_{\Gamma}}\log{t_{P_j}}
\left(\sum_{k=0}^{n}\frac{(-1)^k\dim(\sigma_{\rho,k})e^{-t\lambda_{\rho,k}^2}}
{\sqrt{4\pi t}}\right).
\end{align*}
Taking the Mellin transform, the Proposition follows. 
\end{proof}
Next, as in \cite{Pf}, for each $P_j\in\mathfrak{P}_{\Gamma}$ 
and for $Y>0$ one can define the regularized analytic torsion 
$T(F_{P_j;\Gamma}(Y),\partial F_{P_j;\Gamma}(Y);E_\rho)$
of $F_{P_j;\Gamma}(Y)$ and the bundle $E_\rho|_{F_{P_j;\Gamma}(Y)}$, where one takes relative boundary conditions. 
For different $Y$, these torsions are compared by the following 
gluing formula.

\begin{lem}\label{Glucusp}
Let $c(n)\in\R$ be as in \cite[equation 15.10]{Pf}. 
Then for $Y_1>0$ and $Y_2>0$ one has 
\begin{align*}
&\log T(F_{P_j;\Gamma}(Y_1),\partial F_{P_j;\Gamma}(Y_1);E_\rho)-
c(n)\vol(\partial F_{P_j;\Gamma}(Y_1))\rk(E_\rho)
\\=
&\log T(F_{P_j;\Gamma}(Y_2),\partial F_{P_j;\Gamma}(Y_2);E_\rho) 
-c(n)\vol(\partial F_{P_j;\Gamma}(Y_2))\rk(E_\rho)\\+&
\sum_{k=0}^{n}(-1)^{k+1}\lambda_{\rho,k}\dim(\sigma_{\rho,k})(\log(Y_2)-\log(Y_1)).
\end{align*}
\end{lem}
\begin{proof}
This follows immediately from \cite[Corollary 15.4, equation (15.11) and Corollary 16.2]{Pf}.
\end{proof}

We let $\overline{X}$ denote the 
Borel-Serre compactification of $X$ and we 
let $\tau_{Eis}(X;E_\rho)$ be the 
Reidemeister torsion of $\overline{X}$ with coefficients in $E_\rho$, 
defined as in \cite[section 9]{Pf}.
For simplicity, we assume that $\Gamma$ is normal in $\Gamma_0$. 
Then for the torsion $T_{X_0}(X;E_\rho)$, the main 
result of \cite{Pf} can be restated as follows.

\begin{prop}\label{Gluerestated}
For the analytic torsion $T_{X_0}(X;E_\rho)$ one has 
\begin{align*}
&\log\tau_{Eis}(X;E_\rho)=\log T_{X_0}(X;E_\rho)-\frac{1}{4}
\sum_{k=0}^{d-1}(-1)^k\log|\lambda_{\rho,k}|\dim H^k(\partial \overline{X};E_\rho)\\
&-\sum_{l=1}^{\kappa(\Gamma_0)}\#\{P_j\in\mathfrak{P}_{\Gamma}\colon \gamma_jP_j\gamma_j^{-1}=P_{0,l}\}\biggl(\log T\left(F_{P_{0,l;\Gamma}}(1),\partial F_{P_{0,l;\Gamma}}(1);E_\rho\right)\\ &-c(n)\vol(\partial F_{P_{0,l;\Gamma}}(1))\rk(E_\rho)\biggr).
\end{align*}
\end{prop}
\begin{proof}
By \cite[Theorem 1.1]{Pf}  
and by Proposition \ref{Compartors} we have
\begin{align*}
&\log\tau_{Eis}(X;E_\rho)=\log T_{X_0}(X;E_\rho)+
\sum_{P_j\in\mathfrak{P}_{\Gamma}}\log(t_{P_j})\left(\sum_{k=0}^n\frac{(-1)^k\dim(\sigma_{\rho,k})\lambda_{\rho,k}}{2}\right)\\
-&\sum_{P_j\in\mathfrak{P}_{\Gamma}}\left(\log T\left(F_{P_j;\Gamma}(1),\partial F_{P_j;\Gamma}(1);E_\rho\right)-c(n)\vol(\partial F_{P_j;\Gamma}(1))\rk(E_\rho)\right)
\\ &-\frac{1}{4}
\sum_{k=0}^{d-1}(-1)^k\log|\lambda_{\rho,k}|\dim H^k(\partial \overline{X};E_\rho)
\end{align*}
where we recall that the regularized analytic torsion used 
in \cite[Theorem 1.1]{Pf} is the torsion denoted $T_{\mathfrak{P}_{\Gamma}}(X;E_\rho)$ here. 
Using that $\Gamma$ is normal in $\Gamma_0$, it easily follows from the definition of $t_{P_j}$ that for each $P_j\in\mathfrak{P}_{\Gamma}$ one has a canonical isometry 
$\iota_{P_j;P_{0,l(j)}}:F_{P_j;\Gamma}(1)\cong F_{P_{0,l(j)};\Gamma}(t_{P_j})$. It is 
easy to see that also $\iota_{P_j;P_{0,l(j)}}^*\left(E_{\rho}|_{F_{P_{0,l(j)};\Gamma}(t_{P_j})}\right)$ 
is isometric to $E_{\rho}|_{F_{P_j;\Gamma}(1)}$. Thus we 
have 
\begin{align*}
&\log T(F_{P_j;\Gamma}(1),\partial F_{P_j;\Gamma}(1);E_{\rho})
-c(n)\vol(\partial F_{P_j;\Gamma}(1))\rk(E_\rho)\\ =&\log T(F_{P_{0,l(j);\Gamma}}(t_{P_j}),\partial F_{P_{0,l(j)};\Gamma}(t_{P_j});E_{\rho})-c(n)\vol(\partial F_{P_{0,l(j)};\Gamma}(t_{P_j}))\rk(E_\rho).
\end{align*}
Applying Lemma \ref{Glucusp}, the 
Proposition follows.

\end{proof}

Although the main topic of this paper is the 
behaviour of cohomological torsion of congruence 
subgroups of Bianchi groups under 
a variation of the local system, 
we now state the following limit 
multiplicity formula for twisted Reidemeister torsion in 
arithmetic hyperbolic congruence towers of arbitrary odd dimension, since 
the latter is an easy corollary.

\begin{corollary}\label{Corollarylimmult}
Let $G:=\SO^0(d,1)$, $d$ odd, and for $q\in\mathbb{N}$ let 
$\Gamma(q):=\ker(G(\Z)\to G(\Z/q\Z)$ denote 
the principal congruence subgroup of level $q$. Let 
$X_q:=\Gamma(q)\backslash \mathbb{H}^d$. Then 
for any $\rho\in\Rep(G)$ with $\rho\neq \rho_{\theta}$ one has
\begin{align*}
\lim_{q\to\infty}\frac{\log\tau_{Eis}(X_q;E_\rho)}{\vol(X_q)}=t^{(2)}_{\mathbb{H}^d}(\rho), 
\end{align*} 
where $t^{(2)}_{\mathbb{H}}(\rho)$ is 
the $L^2$-invariant associated to $\rho$ and $\mathbb{H}^d$ 
which is defined as in \cite{BV}, \cite{MPtors} and which is never zero.  
The same holds for ever sequence $X_{\aL}$ of arithmetic 
hyperbolic 3-manifolds associated to princiapl congruence
subgroups $\Gamma(\aL)$ of Bianchi groups if $N(\aL)\to\infty$.  
\end{corollary}
\begin{proof}
First we assue that $G=\SO^0(d,1)$. Let $\Gamma_0:=G(\Z)$ and 
$X_0:=\Gamma_0\backslash\mathbb{H}^d$. By \cite[Corollary 1.3]{MPtors}, for the analytic torsion 
$T_{X_0}(X_q;E_\rho)$ one has 
\begin{align}\label{thrmMP}
\lim_{q\to\infty}\frac{\log T_{X_0}(X_q;E_\rho)}{\vol(X_q)}
=t_{\mathbb{H}^d}^{(2)}(\rho),
\end{align}
as $q\to\infty$. 
Next it is well-known that for the number $\kappa(\Gamma(q))$ of cusps of $\Gamma(q)$ one has 
\begin{align}\label{numbercusps}
\lim_{q\to\infty}\frac{\kappa(\Gamma(q))}{\vol(X_q)}=0,
\end{align}
see for example \cite[Proposition 8.6]{MPtors}.
For 
$P_{0,l}\in\mathfrak{P}_{\Gamma_0}$ we let $\Lambda_{\Gamma(q)}(P_{0,l}):=\log((\Gamma(q)\cap N_{P_{0,l}})$, which is a lattice in $\mathfrak{n}_{P_{0,l}}$. 
By a result of Deitmar and Hoffmann 
\cite[Lemma 4]{DH}, for each 
$P_{0,l}\in\mathfrak{P}_{\Gamma_0}$, 
there exists a finite set of lattices 
$\mathcal{L}_{P_{0,l}}=\{\Lambda_1(P_{0,l}),\dots,\Lambda_m(P_{0,l})\}$ in $\mathfrak{n}_{P_{0,l}}$ 
such that for each $q\in\mathbb{N}$ the lattice 
$\Lambda_{\Gamma(q)}(P_{0,l})$ 
arises by scaling one of the lattices $\Lambda_j(P_{0,l})$, 
$j=1,\dots,m$, see \cite[Lemma 10.1]{MPtors}. 
For $\Lambda_j(P_{0,l})\in\mathcal{L}_{P_{0,l}}$ we let 
$T_{\Lambda_j(P_{0,l}}:=\Lambda_j(P_{0,l})\backslash\mathfrak{n}_{P_{0,l}}$, equipped with the flat metric \eqref{metr} restricted to $\mathfrak{n}_{P_{0,l}}$ which we shall 
denote by $g_{\Lambda_j(P_{0,l})}$. 
Then we let  $F_{\Lambda_j(P_{0,l})}(1):=[1,\infty)\times\Lambda_j(P_{0,l})$, equipped with 
the metric $y^{-2}(dy^2+g_{\Lambda_j(P_{0,l})})$, $y\in [1,\infty)$.
If $\Lambda_{\Gamma(q)}(P_{0,l})=\mu_{q;P_{0,l}}\Lambda_j(P_{0,l})$ 
with $\Lambda_j(P_{0,l})\in\mathcal{L}_{P_{0,l}}$ and
$\mu_{q;P_{0,l}}\in (0,\infty)$, then by Lemma \ref{Glucusp} one has  
\begin{align*}
&\log T\left(F_{P_{0,l;\Gamma(q)}}(1),\partial F_{P_{0,l;\Gamma(q)}}(1);E_\rho\right)-c(n)\vol(\partial F_{P_{0,l;\Gamma(q)}}(1))\rk(E_\rho)\\=&\log T\left(F_{\Lambda_j(P_{0,l})}(1),\partial F_{\Lambda_j(P_{0,l})}(1);E_\rho\right)-c(n)\vol(\Lambda_j(P_{0,l}))\rk(E_\rho)
\\&+\log{\mu_{q;P_{0,l}}}\sum_{k=0}^{n}(-1)^k\lambda_{\rho,k}\dim(\sigma_{\rho,k}). 
\end{align*}

We can obviously estimate
$
\mu_{q;P_{0,l}}\leq C_1[\Gamma_0\cap N_{P_{0,l}}:\Gamma(q)\cap N_{P_{0,l}}]$, 
where $C_1$ is a constant which is independent of $q$. 
Thus there exists a constant $C_2$ such that for all $q$ one can estimate
\begin{align*}
&\left|\log T\left(F_{P_{0,l;\Gamma(q)}}(1),\partial F_{P_{0,l;\Gamma(q)}}(1);E_\rho\right)-c(n)\vol(\partial F_{P_{0,l;\Gamma(q)}}(1))\rk(E_\rho)\right|\\ &\leq 
C_2\log[\Gamma_0\cap N_{P_{0,l}}:\Gamma(q)\cap N_{P_{0,l}}].
\end{align*}

For each $l\in\{1,\dots,\kappa(\Gamma_0)\}$ one has 
\[
\frac{\#\{P_j\in\mathfrak{P}_{\Gamma(q)}\colon \gamma_jP_j\gamma_j^{-1}=P_{0,l}\}}{[\Gamma_0:\Gamma(q)]}
=\frac{\#(\Gamma(q)\backslash\Gamma_0/(\Gamma_0\cap P_{0,l}))}{[\Gamma_0:\Gamma(q)]}\leq \frac{1}{[\Gamma_0\cap N_{P_{0,l}}:\Gamma(q)\cap N_{P_{0,l}}]}.
\]

Thus for all $q$ one can estimate
\begin{align}\label{lasteq}
&\sum_{l=1}^{\kappa(\Gamma_0)}\frac{\#\{P_j\in\mathfrak{P}_{\Gamma(q)}\colon \gamma_jP_j\gamma_j^{-1}=P_{0,l}\}}
{[\Gamma_0:\Gamma(q)]}\nonumber\\ &\times
|\log T\left(F_{P_{0,l;\Gamma(q)}}(1),\partial F_{P_{0,l;\Gamma}}(1);E_\rho\right)-c(n)\vol(\partial F_{P_{0,l;\Gamma(q)}}(1))\rk(E_\rho)|
\nonumber\\ &\leq C_2 \sum_{l=1}^{\kappa(\Gamma_0)}\frac{\log[\Gamma_0\cap N_{P_{0,l}}:\Gamma(q)\cap N_{P_{0,l}}]}{[\Gamma_0\cap N_{P_{0,l}}:\Gamma(q)\cap N_{P_{0,l}}]}.
\end{align}
Since $\cap_q\Gamma(q)=\{1\}$, $[\Gamma_0\cap N_{P_{0,l}}:\Gamma(q)\cap N_{P_{0,l}}]$ as 
$q\to\infty$ and thus the last 
term in \eqref{lasteq} goes to zero as $q$ tends to infinity. 
Since $\dim H^k(\partial\overline{X_q})=O(\kappa(\Gamma_q))$, 
the corollary follows by applying Proposition \ref{Gluerestated}, 
equation \eqref{thrmMP} and equation \eqref{numbercusps}. For 
a sequence $X_{\aL}$ associated to 
principal congruence subgroups of Bianchi groups, one 
argues in the same way. 
\end{proof}


\section{Analytic and combinatorial torsion for congruence subgroups 
of Bianchi groups}\label{sectorsB}           
\label{tors_cong}
\setcounter{equation}{0}
From now on, we let $G=\Spin(3,1)=\SL_2(\C)$, $K=\Spin(3)=\SU(2)$.
We take $P$ to be the standard parabolic subgroup of $G$ consiting 
of all upper triangular matrices in $G$. Then the unipotent 
radical $N_P$ of $P$ is given by all upper triangular 
matrices whose diagonal entries are one. Moreover, we let $M_P$ and $A_P$ denote the subgroups of $\SL_2(\C)$ defined by
\begin{align}\label{MA}
M_P:=\left\{\begin{pmatrix}e^{i\theta}&0\\ 0&e^{-i\theta}\end{pmatrix}\colon\theta\in[0,2\pi]\right\};\quad
A_P:=\left\{\begin{pmatrix}t&0\\ 0&t^{-1}\end{pmatrix}\colon t\in(0,\infty)\right\}.
\end{align}
Then $P=M_PA_PN_P$. 
Let $\nL_P$ denote the Lie algebra of $N_P$ and let $\aL_P$ denote 
the Lie algebra of $A_P$.
For $k\in\mathbb{Z}$, $\lambda\in\C$ we let $\sigma_k:M_P\to\C$, $\xi_\lambda:A_P\to\mathbb{C}$ be defined 
by 
\[
\sigma_k\left(\begin{pmatrix}e^{i\theta}&0\\ 0&e^{-i\theta}\end{pmatrix}\right):=e^{ik\theta};\quad 
\xi_{\lambda}\left(\begin{pmatrix}t&0\\ 0&t^{-1}\end{pmatrix}\right):=t^{2\lambda}. 
\]
Then the assignment $\lambda\to \xi_\lambda$ is consistent 
wit our earlier identification $\C\cong(\aL_P^*)_\C$.  The group $K_\infty$ acts on 
$\gL/\kL$ by $\Ad$ and we let $K_\infty$ act on
$\aL_P\oplus\nL_P$ by using the canonical 
identification $\gL/\kL\cong\aL_P\oplus\nL_P$.
\newline

Let $F:=\Q(\sqrt{-D})$ be 
an imaginary quadratic number field and let $d_F$ be 
its class number. Let $\mathcal{O}_D$
denote the ring of integers of $F$. 
Let $\mathcal{O}_D^*$ be the group of units 
of $\mathcal{O}_D$, i.e. $\mathcal{O}_D^*=\{\pm 1\}$ for $D\neq 1,3$, 
$\mathcal{O}_D^*=\{\pm 1,\pm\sqrt{-1}\}$ for $D=1$, $\mathcal{O}_D^*=\{\pm
1,\pm\frac{1\pm\sqrt{-3}}{2}\}$ for $D=3$.
Let $\Gamma_D:=\SL_2(\mathcal{O}_D)$. The quotient 
$X_0:=\Gamma_D\backslash\mathbb{H}^3$
is a hyperbolic orbifold of finite volume, see for example \cite{EGM}.
We have 
\begin{align}\label{NP}
\Gamma_D\cap N_{P}=\left\{\begin{pmatrix} 1&b\\ 0&1\end{pmatrix}\colon b\in\mathcal{O}_D\right\}
\end{align}

We fix a set $\mathfrak{P}_{\Gamma_D}=\{P_{0,1},\dots,P_{0,\kappa(\Gamma_D)}\}$ of 
representatives of $\Gamma_D$-cuspidal parabolic subgroups of $G$, 
where we require each parabolic to be $F$-rational and where 
we assume that $P_{0,1}=P$. 
Let $\mathbb{P}^1(F)$ be the one-dimensional projective space of $F$. As usual,
we 
write $\infty$ for the element $[1,0]\in \mathbb{P}^1(F)$.
Then $\SL_2(F)$ acts transitively on $\mathbb{P}^1(F)$ and 
$P$ is the stabilizer of $\infty$. In particular 
for each subgroup $\Gamma$ of $\Gamma_D$ 
one has $\kappa(\Gamma)=\#\left(\Gamma\backslash G/P\right)=\#\left(\Gamma\backslash\mathbb{P}^1(F)\right)$. 
Moreover, by \cite[Chapter 7.2, Theorem 2.4]{EGM}, one has $\kappa(\Gamma_D)=d_F$.
Let $\{\eta_l\colon l=1,\dots,\eta_{d_F}\}$ denote fixed representatives 
of $\Gamma_D\backslash\mathbb{P}^1(F)$ such that $P_{0,l}\in\mathfrak{P}_{\Gamma_0}$ 
is the stabilizer of $\eta_l$ in $\SL_2(\C)$ for each $l$.

For $\aL$ a non-zero ideal of $\mathcal{O}_D$ 
we let $\Gamma(\aL)$ denote 
the principal congruence subroup of $\Gamma_D$ 
of level $\aL$. This group is normal in $\Gamma_D$; 
moreover, for $N(\aL)$ sufficiently large ($N(\aL)\geq 3$ 
in the case $\mathcal{O}_D^*=\pm 1$), the group 
$\Gamma(\aL)$ is neat. We 
shall assume from now on that this is the case. Then 
$X_\aL:=\Gamma(\aL)\backslash \mathbb{H}^3$ is a hyperbolic 3-manifold of finite volume.
According 
to the previous section, for 
$\rho\in\Rep(G)$ with $\rho\neq\rho_\theta$, we can 
define the analytic torsion 
$\log T_{X_0}(X_{\aL};E_\rho)$ of $X_{\aL}$ with coefficients 
in $E_\rho$ with respect to the choice of truncation 
parameters coming from the covering $\pi:X_{\aL}\to X_0$ and 
the choice of $\mathfrak{P}_{\Gamma_D}$.

We shall now simplify the formula in  \eqref{Gluerestated} a bit further 
for the specific manifolds $X_\aL$. Let $\mathfrak{b}$ be a 
non-zero ideal of $\mathcal{O}_D$. 
Taking the 
identification \eqref{NP}, we shall regard $\mathfrak{b}$ 
as a lattice in $\mathfrak{n}_{P}$. We denote this latttice  by $\Lambda_{P}(\mathfrak{b})$ 
and we let $T_{\Lambda_P}(\mathfrak{b}):=\exp(\Lambda_{P}(\mathfrak{b}))\backslash N_P$ 
denote the corresponding torus. 
As above, for $r>0$ we let $F_{\Lambda_{P}(\mathfrak{b})}(r):=[r,\infty)\times T_{\Lambda_{P}(\mathfrak{b})}$ 
denote the corresponding cusp.  
We fix ideals $\mathfrak{s}_l$, $l=1,\dots,d_F$ in 
$\mathcal{O}_D$ which represent the class group of $F$.
Then we have:

\begin{prop}\label{propetbianchi}
There exist $n_{1,\Gamma(\aL)},\dots,n_{d_F,\Gamma(\aL)}\in\mathbb{N}\cup\{0\}$,  
with $n_{1,\Gamma(\aL)}+\dots+n_{d_F,\Gamma(\aL)}=d_F$, and there exist $\mu_{1,\Gamma(\aL)},\dots,\mu_{d_F,\Gamma(\aL)}\in (0,\infty)$ such that 
\begin{align*}
\log\tau_{Eis}(X_\aL;E_\rho)=&\log T_{X_0}(X_\aL;E_\rho)
-
\sum_{l=1}^{d_F}\frac{[\Gamma_D:\Gamma(\aL)]}{\#(\mathcal{O}_D^*) N(\aL)}\biggl(
n_{l,\Gamma(\aL)}\log T\left(F_{\Lambda_{P}(\mathfrak{s}_l)}(1),\partial F_{\Lambda_{P}(\mathfrak{s}_l)}(1);E_\rho\right)\\ &-n_{l,\Gamma(\aL)}c(1)\vol\left(\Lambda_{P}(\mathfrak{s}_l)\right)\rk(E_\rho)
+(\lambda_{\rho,1}-\lambda_{\rho,0})\mu_{l,\Gamma(\aL)}\log{\mu_{l,\Gamma(\aL)}}
\biggr)\\  &-\frac{\kappa(\Gamma(\aL))}{2}
(\lambda_{\rho,0}-\lambda_{\rho,1})
.
\end{align*}
Here the $n_{1,\Gamma(\aL)},\dots,n_{d_F,\Gamma(\aL)}\in\mathbb{N}\cup\{0\}$ and 
the $\mu_{1,\Gamma(\aL)},\dots,\mu_{d_F,\Gamma(\aL)}\in (0,\infty)$ depend on $\Gamma(\aL)$, but not on the representation $\rho$. 
\end{prop}
\begin{proof}
Let $\{\eta_l\colon l=1,\dots,\eta_{d_F}\}$ denote fixed representatives 
of $\Gamma_D\backslash\mathbb{P}^1(F)$ such that $P_{0,l}\in\mathfrak{P}_{\Gamma_0}$ 
is the stabilizer of $\eta_l$ in $\SL_2(\C)$ for each $l$. 
For each $l=1,\dots,d_F$ we fix $B_{\eta_{l}}\in\SL_2(F)$ with
$B_{\eta_{l}}\eta_{l}=\infty$. 
Then $P_{0,l}=:P_{\eta_{l}}=B_{\eta_{l}}^{-1}PB_{\eta_{l}}$. Let 
$(\Gamma_D)_{\eta_{l}}=\Gamma_D\cap P_{\eta_l}$ resp. $\Gamma(\aL)_{\eta_{l}}=\Gamma(\aL)\cap P_{\eta_l}$ be the stabilizers of $
\eta_{l}$ 
in $\Gamma_D$ resp. $\Gamma(\aL)$.
One has
\begin{align}\label{eqB}
B_{\eta_{l}}(\Gamma_D)_{\eta_{l}} B_{\eta_{l}}^{-1}=\left\{J(B_{\eta_{l}}(\Gamma_D)_\eta B_{\eta_{l}}^{-1}\cap N)\colon
J\in\left\{\begin{pmatrix}\alpha&0\\ 0&\alpha^{-1}
\end{pmatrix},\alpha\in\mathcal{O}_D^*\right\}\right\}. 
\end{align}
In particular, one has $B_{\eta_{l}}\Gamma(\aL)_{\eta}B_{\eta_{l}}^{-1}\cap P=B_{\eta_{l}}\Gamma(\aL)_{\eta}B_{\eta_{l}}^{-1}\cap N$ for $N(\aL)$ sufficiently large. 
Write $B_{\eta_{l}}=\begin{pmatrix}\alpha_{l} &\beta_l \\ \gamma_l &\delta_l\end{pmatrix}\in
\SL_2(F)$ and let $\mathfrak{u}_l$ be the $\mathcal{O}_D$-module 
generated by $\gamma_l$ and $\delta_l$. Then one has
\begin{align*}
B_{\eta_{l}}(\Gamma_D)_{\eta_{l}} B_{\eta_{l}}^{-1}\cap N=\left\{\begin{pmatrix}1&\omega\\
0&1\end{pmatrix};\: \omega\in \mathfrak{u}_l^{-2}\right\};\:
B_{\eta_{l}}\Gamma(\aL)_{\eta}B_{\eta_{l}}^{-1}\cap N=\left\{\begin{pmatrix}1&\omega'\\
0&1\end{pmatrix};\: \omega'\in \aL\mathfrak{u}_l^{-2}\right\},
\end{align*}
where the first equality is proved in \cite[Chapter 8.2, Lemma 2.2]{EGM} and
where the second equality
can be proved using the same arguments. 
Thus one has 
\begin{align*}
[B_{\eta_{l}}(\Gamma_D)_{\eta_{l}} B_{\eta_{l}}^{-1}\cap N:B_{\eta_{l}}\Gamma(\aL)_{\eta_{l}} B_{\eta_{l}}^{-1}\cap
N]=N(\aL).
\end{align*}
Thus by \eqref{eqB}, for each $l=1,\dots,d_F$ and $N(\aL)$ suffciently
large 
one has $[(\Gamma_D)_{\eta_{l}}:\Gamma(\aL)_{\eta_{l}}]=\#(\mathcal{O}_D^*)N(\aL)$ and so 
one has 
\begin{align*}
\#\{P_j\in\mathfrak{P}_{\Gamma(\aL)}\colon \gamma_jP_j\gamma_j^{-1}=P_{0,l}\}=\frac{[\Gamma_D:\Gamma(\aL)]}{\#(\mathcal{O}_D^*) N(\aL)}
\end{align*}

For each $l$ there is a constant $\kappa_l>0$ which depends only 
on $B_{\eta_l}$ such that one has a canonical isometry $\iota: F_{P_{0,l};\Gamma(\aL)}(1)\cong F_{\Lambda_{P}(\kappa_l\aL\mathfrak{u}_l^{-2})}(1)$ 
which induces an isomtery $\iota^* E_\rho|_{F_{\Lambda_{P}(\kappa_l\aL \mathfrak{u}_l^{-2})}(1)}\cong E_\rho|_{F_{P_{0,l};\Gamma(\aL)}(1)}$. Next, there exists a unique map $\sigma$ from $\{1,\dots,d_F\}$ into itself
such that for each $l=1,\dots,d_F$ there 
exists a $\tilde{\mu}_{l,\Gamma(\aL)}\in F^*$ with $\aL\mathfrak{u}_{l}^{-2}=\tilde{\mu}_{l,\Gamma(\aL)}\mathfrak{s}_{\sigma(l)}$. 
We let $\mu_{l,\Gamma(\aL)}:=\kappa_l|\tilde{\mu}_{l,\Gamma(\aL)}|$.
Then it follows that there is a canonical isometry $\iota:F_{P_{0,l};\Gamma(\aL)}(1)\cong F_{\Lambda_{P}(\mathfrak{s}_{\sigma(l)})}(\mu_{l,\Gamma(\aL)}^{-1})$ 
which extends to an isometry $\iota^*E_\rho|_{F_{\Lambda_{P}(\mathfrak{s}_{\sigma(l)})}(\mu_{l,\Gamma(\aL)}^{-1})}\cong E_\rho|_{F_{P_{0,l};\Gamma(\aL)}(1)}$. 
Thus it follows from Lemma \ref{Glucusp} that
\begin{align*}
&\log T(F_{P_{0,l};\Gamma(\aL)}(1),\partial F_{P_{0,l};\Gamma(\aL)}(1);E_\rho)-c(1)
\vol(\partial F_{P_{0,l};\Gamma(\aL)}(1))\rk(E_\rho)\\
=&\log T\left(F_{\Lambda_{P}(\mathfrak{s}_{\sigma(l)})}(1),\partial F_{\Lambda_{P}(\mathfrak{s}_{\sigma(l)})}(1);E_\rho\right)-c(1)\vol\left(\partial F_{\Lambda_{P}(\mathfrak{s}_{\sigma(l)})}(1)\right)
+\log{\mu_{l,\Gamma(\aL)}}(\lambda_{\rho,1}-\lambda_{\rho,0}).
\end{align*} 
For each $l=1,\dots,d_F$ we let $n_{l,\Gamma(\aL)}:=\#\{\sigma^{-1}(l)\}$.
Then, since $\dim(\sigma_{\rho,k})=1$ in the case of $G=\SL_2(\C)$, the 
Proposition follows from Proposition \ref{Gluerestated}.
\end{proof}

We let $\rho_1$ denote the standard representation 
of $\SL_2(\C)$ on $\mathbb{C}^2$ and for $m\in\mathbb{N}$ we 
consider the representation $\rho_m:=\Symm^m\rho_1$ on $V(m):=\Symm^m\mathbb{C}^2$. 
We can now deduce the following result on the 
growth of the torsion $\mathcal{\tau}_{Eis}(X_\aL;E(\rho_m))$ if $m\to\infty$. .

\begin{prop}\label{asRT}
Let $\aL$ and $\aL_0$ be two ideals in $\mathcal{O}_D$ such that 
$\Gamma(\aL)$ and $\Gamma(\aL_0)$ are neat and such that, 
in the notation of the previous proposition, one has 
$n_{l,\Gamma(\aL)}=n_{l,\Gamma(\aL_0)}$ for each 
$l=1,\dots,d_F$. 
Then one has
\begin{align*}
&\frac{[\Gamma_D:\Gamma(\aL_0)]}{\#(\mathcal{O}_D^*) N(\aL_0)}\log\tau_{Eis}(X_{\aL};E(\rho_m))- 
\frac{[\Gamma_D:\Gamma(\aL)]}{\#(\mathcal{O}_D^*) N(\aL)}\log\tau_{Eis}(X_{\aL_0};E(\rho_m))
\\=&-\frac{[\Gamma_D:\Gamma(\aL_0)][\Gamma_D:\Gamma(\aL)]}{\#\mathcal{O}_D^*\pi}\left(\frac{1}{N(\aL_0)}-\frac{1}{N(\aL)}\right)\vol(\Gamma_D\backslash\mathbb{H}^3)m^2+O(m\log{m}),
\end{align*}
as $m\to\infty$. 
\end{prop}
\begin{proof}
In the notation of \cite{MPtors1}, the representation 
$\rho_m$ is the representation of highest weight 
$(m/2,m/2)$. 
If $\Gamma$ is a neat finite index subgroup of $\Gamma_D$ and if we let $X:=\Gamma\backslash\mathbb{H}^3$, then 
specializing \cite[Theorem 1.1]{MPtors1} to the 
present situation, we obtain  
\begin{align*}
\lim_{m\to\infty}\log T_{X_0}(X;E_{\rho_m})=-\frac{1}{\pi}\vol(X)m^2+O(m\log{m}),
\end{align*}
as $m\to\infty$. Moreover, one has $\lambda_{\rho_m,0}=(m+1)/2$ and $\lambda_{\rho_m,1}=m/2$. 
Thus the proposition follows from the previous Proposition \ref{propetbianchi}. 
\end{proof}
There is a constant $C_1'(\Gamma)$ such that $\Gamma(\aL)$ 
is neat for all ideals $\aL$ of $\mathcal{O}_D$ with $N(\aL)\geq C_1'(\Gamma)$. If 
$\mathcal{O}_D^*=\{\pm 1\}$, one has $C_1'(\Gamma)=3$.
Next we remark that by the requirement $n_{1,\Gamma(\aL)}+\dots+n_{d_F}(\aL)=d_F$ there 
is a finite set $\mathcal{A}$ of ideals of $\mathcal{O}_D$ such 
that $N(\aL_0)\geq C_1'(\Gamma)$ for each $\aL_0\in\mathcal{A}$ and  
such that for each non-zero ideal $\aL$ of $\mathcal{O}_D$ with 
$N(\aL)\geq C_1'(\Gamma)$ there 
is an $\aL_0\in\mathcal{A}$ such that  $n_{l,\Gamma(\aL)}=n_{l,\Gamma(\aL_0)}$ for each 
$l=1,\dots,d_F$. We let $\tilde{C}_1(\Gamma):=\max\{N(\aL_0)\colon\aL_0\in\mathcal{A}\}$ 
and we let $C_1(\Gamma):=\max\{\tilde{C}_1(\Gamma),C_1'(\Gamma)\}$.


\section{Torsion in cohomology and Reidemeister torsion}\label{secTc}
\label{comp_tors}
\setcounter{equation}{0}
We keep the notation of the preceding section. 
Let $\Lambda(m):=\Symm^m\mathcal{O}_D$. Then $\Lambda(m)$ is a lattice in 
$V(m)$ which is preserved by $\Gamma(D)$; we denote 
the representation of $\Gamma(D)$ on $\Aut_{\mathbb{Z}}(\Lambda(m))$ 
by $\rho_{m;\mathbb{Z}}$. 
Let $\check{\Lambda}(m):=\Hom_{\mathbb{Z}}(\Lambda(m),\mathbb{Z})$ 
denote the dual lattice of $\Lambda(m)$ and let $\check{\rho}_{m;\mathbb{Z}}$ denote 
the contragredient representations of $\Gamma(D)$ on $\check{\Lambda}_m$. 
We let $\bar{\rho}_{\mathbb{Z};m}:=\rho_{\mathbb{Z};m}\oplus\check{\rho}_{\mathbb{Z};m}$ 
denote the corresponding integral representation of $\Gamma(D)$ on $\bar{\Lambda}(m):=\Lambda(m)\oplus\check{\Lambda}(m)$. 
We let $\bar{\rho}_m$ denote the corresponding 
real representation of $\Gamma$ on $\bar{V}(m):=V(m)\oplus V(m)^*$. 
Over $\R$, the representation 
$\rho_m$ is self-dual, i.e. one has 
$\bar{\rho}_m\cong\rho_m\oplus\rho_m$. In particular, no irreducible summand 
of $\bar{\rho}_m$ is invariant 
under the Cartan involution $\theta$. 
Let $\Gamma$ be any neat, finite 
index subgroup of $\Gamma(D)$. Then we regard each $\bar{\Lambda}(m)$ 
as a $\Gamma$-module. Let $H^*(\Gamma,\bar{\Lambda}(m))$ 
be the cohomology groups of $\Gamma$ with coefficients in $\bar{\Lambda}(m)$. 
These groups are finitely generated abelian groups and thus they admit a decomposition
\[                                                         
H^*(\Gamma,\bar{\Lambda}(m))=H^*(\Gamma,\bar{\Lambda}(m))_{free}\oplus H^*(\Gamma,\bar{\Lambda}(m))_{tors},
\]
where $H^*(\Gamma,\bar{\Lambda}(m))_{free}$ are free, finite-rank $\Z$-modules 
and where $H^*(\Gamma,\bar{\Lambda}(m))_{tors}$ are finite abelian groups. 
Let $X:=\Gamma\backslash\mathbb{H}^3$ and let $\overline{X}$ denote 
the Borel-Serre compactification of $X$. The latter is homotopy equivalent 
to $X$. In particular, $\bar{\rho}_{\mathbb{Z},m}$ defines an integral 
local system $\mathcal{L}(m)$ over $\overline{X}$. Since the universal covering 
of $X$ resp. $\overline{X}$ is contractible, one has 
a canonical isomorphism $H^*(\Gamma,\bar{\Lambda}(m))\cong H^*(\overline{X},\mathcal{L}(m))$. 
Let $E_{\bar\rho_m}$ denote the flat vector bundle over $X$ resp. $\overline{X}$ 
corresponding to $\bar\rho_m$ and let 
$H^*(\overline{X};E_{\bar\rho_m})$ denote 
the singular cohomology groups of $\overline{X}$ with 
coefficients in $E_{\bar\rho_m}$. Then $H^*(\Gamma,\bar{\Lambda}(m))_{free}$ is a lattice in $H^*(\overline{X};E_{\bar\rho_m})$.

 We now recall the description of the canonical bases in the 
cohomology $H^*(\overline{X};E_{\bar{\rho}_m})$ which are used 
to define the Reidemeister torsion $\tau_{Eis}(X;E_{\bar{\rho}_m})$. 
For more details, we refer to \cite{Pf}.  We shall use the 
notation of \cite[section 8]{Pf}. Let $\partial\overline{X}$ denote 
the boundary of $\overline{X}$. If $\iota:\partial\overline{X}\to X$ 
denotes the inclusion map, the corresponding maps $\iota_k^*:H^k(\overline{X};E_{\bar{\rho}_m})\to H^k(\partial\overline{X};_{\bar{\rho}_m})$ in cohomology are injective for  $k\in\{1,2\}$, see \cite[Lemma 8.3]{Pf}.  
Thus the cohomology $H^*(\overline{X};E_{\bar{\rho}_m})$ is 
completely described in terms of Eisenstein cohomolgy due to Harder \cite{Ha} .  
For each $P_j\in\mathfrak{P}_\Gamma$ let $\mathcal{H}^k(\mathfrak{n}_{P_j};\bar{V}(m))$ denote 
the harmonic forms of degree $k$ in the Lie algebra cohomology complex of $\mathfrak{n}_{P_j}$ 
with coefficients in $\bar{V}(m)$. 
We equip this space with the inner product 
induced by the restriction of the inner product \eqref{metr} on $\gL$ to
$\nL_{P_j}$ and the admissible 
inner product on $\bar{V}(m)$.    
Let $\sigma^{-}_{\bar{\rho}_m,1}\in\hat{M}_P$ and $\lambda_{\bar{\rho}_m,1}^-\in(-\infty,0)$
resp. $\sigma_{\bar{\rho}_m,2}\in\hat{M}_P$  and  
$\lambda_{\bar{\rho}_m,2}\in(-\infty,0)$
be defined as in \cite[section 6]{Pf}. 
In the present situation, we have $\sigma^{-}_{\bar{\rho}_m,1}=\sigma_{-m-2}$ and 
$\lambda_{\bar{\rho}_m,1}^-=-m/2$ resp. $\sigma_{\bar{\rho}_m,2}=\sigma_{-m}$  and  
$\lambda_{\bar{\rho}_m,2}=-(m+1)/2$.
By the finite-dimensional Hodge thereom and a theorem of van Est one has a canonical 
isomorphism
\begin{align}\label{VE}
H^k(\partial\overline{X};E_{\bar{\rho}_m})\cong\bigoplus_{P_j\in\mathfrak{P}_{\Gamma}}\mathcal{H}^k(\mathfrak{n}_{P_j},\bar{V}(m)).
\end{align}
In degree $1$, Kostant's theorem gives a splitting 
$\mathcal{H}^1(\mathfrak{n}_{P_j},\bar{V}(m))=\mathcal{H}^1(\mathfrak{n}_{P_j},\bar{V}(m))_{-}\oplus
\mathcal{H}^1(\mathfrak{n}_{P_j},\bar{V}(m))_{+}$. 
We let
\[
H^1(\partial\overline{X};E_{\bar{\rho}_m})_{\pm}:=\bigoplus_{j=1}^{\kappa(\Gamma)}\mathcal{H}^1(\mathfrak{n}_{P_j},\bar{V}(m))_{\pm}.
\]
Then out of the constant term matrix of the Eisenstein series one obtains 
a map 
\begin{align*}
\underline{\mathbf{C}}
(\sigma_{-m-2},m/2):H^1(\partial\overline{X};E_{\bar{\rho}_m})_{-}\to H^1(\partial\overline{X};E_{\bar{\rho}_m})_{+}.
\end{align*}
We have $\dim\mathcal{H}^2(\nL_{P_j};\bar{V}(m))=2\dim\mathcal{H}^2(\nL_{P_j};V(m))=2$ 
as well as $\dim\mathcal{H}^1(\nL_{P_j};\bar{V}(m))_{\pm}=2\dim\mathcal{H}^1(\nL_{P_j};V(m))_{\pm}=2$. 
For each $P_j\in\mathfrak{P}_\Gamma$ let
let $\Phi_{i,j}^1$ be an orthonormal basis of
$\mathcal{H}^1(\nL_{P_j};\bar{V}(m))_{-}$ and let $\Phi_{i,j}^2$ be an orthonormal basis of $\mathcal{H}^2(\nL_{P_j};\bar{V}(m))$.
Then the set  
\[
\mathcal{B}^1(\Gamma;\bar\rho_m):=\{E(\Phi_{i,j}^1,m/2)\colon
j=1,\dots,\kappa(\Gamma)\colon i=1,\dots,\dim
\mathcal{H}^n(\nL_{P_j};\bar{V}(m))_{-}\}
\]
forms a basis of $H^1(\overline{X};E_{\bar{\rho}_m})$, 
where $E(\Phi_{i,j}^n,m/2)$ denotes 
again the Eisenstein series evaluated 
at $m/2$ as in \cite[(7.3)]{Pf} which 
is regular at this point by \cite[Proposition 8.4]{Pf}. 
Moreover, in degree 2 
the set 
\[
\mathcal{B}^2(\Gamma;\bar\rho_m):=\{E(\Phi_{i,j}^2,(m+1)/2)\colon
j=1,\dots,\kappa(\Gamma)\colon i=1,\dots,\dim
\mathcal{H}^2(\nL_{P_j};\bar{V}(m))\}
\]
forms a basis of $H^2(\overline{X};E_{\bar{\rho}_m})$. 
Here $E(\Phi_{i,j}^2,(m+1)/2)$ denotes 
the Eisenstein series associated to $\Phi_{i,j}^2$ 
evaluated at $(m+1)/2$ which is again regular 
at this point. 
For $\Phi_{i,j}^1\in\mathcal{H}^1(\nL_{P_j};\bar{V}(m))_{-} $ one has 
\begin{align}\label{restrictionI}
\iota_1^* E(\Phi_{i,j}^1,m/2)=\Phi_{i,j}^1+
\underline{\mathbf{C}}(\sigma_{-m-2},m/2)\Phi_{i,j}^1
\end{align}
and for $\Phi_{i,j}^2\in\mathcal{H}^2(\nL_{P_j};\bar{V}(m))$ one has
\begin{align}\label{restrictionII}
\iota_2^*E(\Phi_{i,j}^2,(m+1)/2)=\Phi_{i,j}^2.
\end{align} 
By the definition of \cite[section 9]{Pf}, the Reidemeister 
torsion $\tau_{Eis}(X;E_{\bar{\rho}_m})$ is taken 
using the above bases in the cohomology. 
In particular, following Bergeron and Venkatesh \cite{BV}, the size of the groups $H^*(\Gamma,\bar{\Lambda}(m))_{tors}$ is related
to the combinatorial torsion $\tau_{Eis}(X;E_{\bar{\rho}_m})$ in the 
following way.
For 
$k\in\{1,2\}$ we let 
$\vol_{\mathcal{B}_k(\Gamma;\bar{\rho}_m)}(H^k(\Gamma;\bar{\Lambda}(m))_{free})$
denote the covolume of the lattice $H^k(\Gamma;\bar{\Lambda}(m))_{free}$
in $H^k(\overline{X};E(\bar\rho_m))$ with respect to 
the inner product which arises by taking the basis $\mathcal{B}_k(\Gamma;\bar{\rho}_m)$ as an orthonormal basis.
Then by \cite[section 2.2]{BV} one has
\begin{align}\label{eqBV}
\log\tau_{Eis}(X;E_{\bar\rho_m})=&\log|H^1(\Gamma;\bar{\Lambda}(m))_{tors}|-
\log|H^2(\Gamma;\bar{\Lambda}(m))_{tors}|\nonumber\\ -&\log\vol_{\mathcal{B}^1(\Gamma;\bar{\rho}_m)}(H^1(X,\bar{\Lambda}(m))_{free}) +\log\vol_{\mathcal{B}^2(\Gamma;\bar{\rho}_m)}(H^2(X,\bar{\Lambda}(m))_{free})
\end{align}

In the notation of \cite{BV}, the term in the second line of \eqref{eqBV} is 
called the regulator. We need to study the regulator 
further. We denote by $H^*(\partial\overline{X};\bar\Lambda(m)$ resp. $H^*(\partial\overline{X};E(\bar\rho_m))$ the cohomology of $\partial\overline{X}$ with coefficients 
in $\mathcal{L}(m)$ resp. $E(\bar\rho_m)$ restricted to $\partial\overline{X}$.  
Again there is a decomposition 
$H^*(\partial\overline{X};\bar{\Lambda}(m))=H^*(\partial\overline{X};\bar{\Lambda}(m))_{free}\oplus H^*(\partial\overline{X};\bar{\Lambda}(m))_{tors}$ and $H^*(\partial\overline{X};\bar{\Lambda}(m))_{free}$ is a lattice in $H^*(\partial\overline{X};E_{\bar\rho_m})$. It is easy to see that $
H^1(\partial\overline{X};\bar{\Lambda}(m))_{\pm}:=H^1(\partial\overline{X};\bar{\Lambda}(m))_{free}\cap H^1(\partial\overline{X};E_{\bar{\rho}_m})_{\pm}$ 
are $\Z$-lattices in $H^1(\partial\overline{X};E_{\bar{\rho}_m})_{\pm}$. 
Thus $H^1(\partial\overline{X};\bar{\Lambda}(m))_-\oplus 
H^1(\partial\overline{X};\bar{\Lambda}(m))_+$ is a $\Z$-sublattice of finite 
index in $H^1(\partial\overline{X};\bar{\Lambda}(m))_{free}$. 
Moreover, with respect to the lattices $H^1(\partial\overline{X};\bar{\Lambda}(m))_\pm$, the matrix 
 $\underline{\mathbf{C}}(\sigma_{-m-2},m/2)$ is $\Q$-rational. 
In the present case where there is no interior cohomology, this 
just follows from \eqref{restrictionI} and the fact that the map $\iota_1^*$ is 
defined over $\Q$.  
In other words, there exists $N\in\mathbb{N}$ such 
that $N\cdot\underline{\mathbf{C}}(\sigma_{-m-2},m/2)$ 
defines a map
\begin{align}\label{denominator}
N\cdot\underline{\mathbf{C}}(\sigma_{-m-2},m/2):
H^1(\partial\overline{X};\bar{\Lambda}(m))_-\to
H^1(\partial\overline{X};\bar{\Lambda}(m))_+.
\end{align}

We let $d_{Eis,\mathbf{C}}(\Gamma,\bar{\rho}_m)$ denote the smallest $N\in\mathbb{N}$ such 
that \eqref{denominator} holds.
We point out that we do not treat the denominator 
of an Eisenstein cohomology class but only the denominator 
of the constant term matrix. However, this is sufficient 
for our purposes due to the following Lemma.

\begin{lem}\label{Lemreg1}
One can estimate 
\begin{align*}
&\log\vol_{\mathcal{B}^1(\Gamma;\bar{\rho}_m)}(H^1(X,\bar{\Lambda}(m)))_{free}\leq \log|H^2(X,\bar{\Lambda}(m)|_{tors}
+\kappa(\Gamma)\log\left(d_{Eis,\mathbf{C}}(\Gamma,\bar{\rho}_m)\right) 
\\ +&\kappa(\Gamma)\log\left([H^1(\partial\overline{X};\bar{\Lambda}(m))_{free}:H^1(\partial\overline{X};\bar{\Lambda}(m))_-\oplus 
H^1(\partial\overline{X};\bar{\Lambda}(m))_+]\right). 
\end{align*}
\end{lem}

\begin{proof}
Let $k:=[H^1(\partial\overline{X};\bar{\Lambda}(m))_{free}:H^1(\partial\overline{X};\bar{\Lambda}(m))_-\oplus 
H^1(\partial\overline{X};\bar{\Lambda}(m))_+]$. We define a free $\Z$-submodule 
$A(\bar{\rho}_m)$ of $H^1(\partial\overline{X};\bar{\Lambda}(m))_-\oplus 
H^1(\partial\overline{X};\bar{\Lambda}(m))_+$ by 
\begin{align*}
&A(\bar{\rho}_m)\\:=&\{\eta_i+\underline{\mathbf{C}}(\sigma_{-m-2},m/2)\eta_i\colon \eta_i \in H^1(\partial\overline{X};\bar{\Lambda}(m))_-\colon \underline{\mathbf{C}}(\sigma_{-m-2},m/2)\eta_i\in H^1(\partial\overline{X};\bar{\Lambda}(m))_+ \}.
\end{align*}
Let $\pr_{free}:H^1(\partial\overline{X};\bar{\Lambda}(m))\to H^1(\partial\overline{X};\bar{\Lambda}(m))_{free}$ be the projection. Then by \eqref{restrictionI} the lattice $k\cdot\pr_{free}i_1^*H^1(\overline{X};\bar{\Lambda}(m))_{free}$ is a $\Z$- sublattice of $A(\bar{\rho}_m)$
whose $\Z$-rank is equal to that of $A(\bar{\rho}_m)$. Thus the quotient 
$ A(\bar{\rho}_m)/\left(k\cdot\pr_{free}i_1^*H^1(\overline{X};\bar{\Lambda}(m))_{free}\right)$ embeds into 
$H^2(\overline{X},\partial \overline{X};\bar{\Lambda}(m))_{tors}$ by the long exact cohomology sequence. 
By Poincar\'e duality, \cite[page 223-224]{Wa} and the universal 
coefficient theorem one has 
\[
H^2(\overline{X},\partial \overline{X};\bar{\Lambda}(m))_{tors}\cong H_1(\overline{X};\bar{\Lambda}(m))_{tors}
\cong H^2(\overline{X};\bar{\Lambda}(m))_{tors},
\]
where in the last isomorphism we used that $\bar{\Lambda}(m)$ 
was self-dual over $\Z$. 
On the other hand, if $\pi_{-}:A(\bar{\rho}_m)\to
H^1(\partial\overline{X};\bar{\Lambda}(m))_-$ is the projection, 
then by the definition of $d_{Eis,\mathbf{C}}(\Gamma,\bar{\rho}_m)$
and the fact that $\rk_\Z H^1(\partial\overline{X};\bar{\Lambda}(m))_-=\kappa(\Gamma)$, the order 
of the quotient  $H^1(\partial\overline{X};\bar{\Lambda}(m))_-/\pi_{-}A(\bar{\rho}_m)$ can be estimated as
\[
|H^1(\partial\overline{X};\bar{\Lambda}(m))_-/\pi_{-}A(\bar{\rho}_m)|\leq d_{Eis,\mathbf{C}}(\Gamma,\bar{\rho}_m)^{\kappa(\Gamma)}. 
\]
Thus the first estimate follows easily from 
the definition of $\vol_{\mathcal{B}^1(\Gamma;\bar{\rho}_m)}(H^1(X,\bar{\Lambda}(m)))_{free}$. 
\end{proof}

\begin{lem}\label{Lemreg2}
We have :
$$
\log\vol_{\mathcal{B}^2(\Gamma;\bar{\rho}_m)}(H^2(X,\bar{\Lambda}(m))_{free} \leq \log|H^1(X,\bar{\Lambda}(m)|_{tors}.
$$
\end{lem}

\begin{proof}
Similar to that of the previous lemma.  
\end{proof}


\section{The adelic intertwining operators}
\label{intertwine_adele}
\setcounter{equation}{0}
In this section and the next one, we want to establish an estimate of $d_{Eis,\mathbf{C}}(\Gamma,\bar{\rho}_m)$ by working adelically. We let $G:=\SL_2$ regarded 
as an algebraic group over $F$. Also, for notational convenience we shall write $K_\infty:=\SU(2)$. Let $\mathbb{A}$ denote the adele ring of $F$ and let $\mathbb{A}_f$ be the finite adeles. For a linear algebraic group $H$ defined over $F$ let $H(\mathbb{A})$ denote its adelic points. For $v$ a finite place we let $F_v$ denote the completion of $F$ at $v$,  we let $\mathcal{O}_v$ denote the integers in $F_v$ and we let $\pi_v\in\mathcal{O}_v$ be a fixed uniformizer. We let $\Gamma$ be a fixed neat congruence subgroup of $\SL_2(\mathcal{O}_D)$ of level $\aL=\prod_{v\:\text{finite}}\mathfrak{p}_v^{n_v}$, where $\mathfrak{p}_v$ is the prime ideal corresponding to $v$. 

Let $P$ be the parabolic subgroup of $G$ consisting of upper triangular matrices and let $T$ denote the set of diagonal matrices of determinant one. 
Let $N_P$ denote the upper triangular matrices with $1$ as diagonal entries. We regard both $P$ and $N_P$ as algebraic groups over $F$. Then $P=TN_P$. Let 
\[
K_{max}:=K_\infty\times\prod_{v \:\text{finite}}\SL_2(\mathcal{O}_v).
\]
Then one has $G(\mathbb{A})=P(\mathbb{A})K_{max}$. Let $K(\Gamma)_f\subseteq K_{max}$ be the compact subgroup of $G(\mathbb{A}_f)$ corresponding to $\Gamma$, i.e. $\Gamma=G(F)\cap K(\Gamma)_f$, where $\Gamma$ and $G(F)$ are embedded diagonally into $G(\mathbb{A}_f)$. Let $g_{P_1},\dots,g_{P_h}\in G(F)$ denote fixed 
representatives of $P(F)\backslash G(F)/\Gamma$. We assume $g_{P_1}=1$. Let $P_i:=g_{P_i}^{-1}Pg_{P_i}$ denote the corresponding parabolic subgroups of $G$ defined over $F$. In this section and the next one, we let $N_{P_i}$ denote the unipotent radical of $P_i$ regarded as an algebraic group over $F$ and we shall denote by $N_{P_i,\infty}$ the corresponding real subgroup of $\SL_2(\C)$. We embed $G(F)$, $P(F)$ as well as the elements $g_{P_1},\dots, g_{P_h}$ diagonally into $G(\mathbb{A})$. Then we have a canonical isomorphism
\begin{align}\label{iso}
\mathcal{I}_{\mathbb{A}}:\:P(F)\backslash G(\mathbb{A})/K(\Gamma)_f\cong \bigsqcup_{i=1}^h(\Gamma\cap N_{P_i,\infty})\backslash \SL_2(\C)
\end{align}
which is defined as follows. By the strong approximation theorem one has $G(\mathbb{A})=G(F)\SL_2(\C) K(\Gamma)_f$. This implies that each $g\in G(\mathbb{A})$ can be written as 
\begin{align}\label{decg}
g=bg_{P_i}g_\infty k_f, 
\end{align}
where $b\in P(F)$, $g_{P_i}\in\{g_{P_1},\dots,g_{P_h}\}$ is uniquely determined and $g_\infty$ is unique up to $\Gamma\cap P_i(F)=\Gamma\cap N_{P_i,\infty}$. Let $\pi:G(\mathbb{A})\to P(F)\backslash G(\mathbb{A})/K(\Gamma)_f$ denote the projection. Then according to \eqref{decg}, for $g\in G(\mathbb{A})$ we set $\mathcal{I}_{\mathbb{A}}(\pi(g)):=(\Gamma\cap N_{P_i,\infty})g_\infty$, where $(\Gamma\cap N_{P_i,\infty})g_\infty$ denotes the equivalence class of $g_\infty$ in $(\Gamma\cap N_{P_i,\infty})\backslash\SL_2(\C)$. We let $\mathcal I_{\mathbb{A},P_i}:P(F)\backslash G(\mathbb{A})/K(\Gamma)_f\to(\Gamma\cap N_{P_i,\infty})\backslash \SL_2(\C)$ be the maps induced by $\mathcal{I}_{\mathbb{A}}$. The map $\mathcal I_\Ade$ induces an isomorphism
\begin{multline}\label{eqI}
\left(\mathcal{I}_\mathbb{A}\right)^* :\: \bigoplus_{i=1}^{\kappa(\Gamma)} \left(C^\infty(N_{P_i,\infty}\backslash\SL_2(\C))\otimes(\nL_P\oplus\aL_P)^*\otimes \bar{V}(m)\right)^{K_\infty} \\
   \xrightarrow[]{\sim} \left(C^\infty(P(F)N(\mathbb{A})\backslash G(\mathbb{A})/K(\Gamma)_f) \otimes (\nL_P\oplus\aL_P)^*\otimes \bar{V}(m)\right)^{K_\infty} =: W.
\end{multline}

Here $K_\infty$ acts on the $C^\infty$-spaces by right translation and on $(\aL_P\oplus\nL_P)^*\otimes \bar{V}(m)$ by $\Ad^*\otimes\rho(m)$. We shall denote this representation also by $\nu_1(\rho(m))$.  We regard the real subgroups $M_P$ and $A_P$ of $\SL_2(\C)$ from \eqref{MA} or more generally the real subgroups $M_{P_i}$ and $A_{P_i}$ for a parabolic $P_i$ introduced above as subgroups of $G(\mathbb{A})$. For $\sigma\in \hat{M}_P$ with $[\nu_1(\rho(m)):\sigma]\neq 0$ and $\lambda\in\mathbb{C}$ we let $W_{\bar{\rho}_m}(\sigma,\lambda)$ be defined by:
\begin{equation} \label{def_space_W}
W_{\bar{\rho}_m}(\sigma,\lambda) = \left\{ f\in W:\: \forall g\in G(\mathbb{A}), a\in A_P,\, m\in M_P,\, f(amg) = \xi_{\lambda+1}(a)\sigma^{-1}(m)f(g) \right\}.
\end{equation}
Let $K_{P_i,\infty}:=g_{P_i}^{-1}K_{\infty}g_{P_i}$. Then $K_{P_i,\infty}$ acts on $(\nL_P\oplus\aL_P)^*\otimes \bar{V}(m)$ by conjugating with $g_{P_i}$. We let $\mathcal{E}_{P_i}(\sigma,\lambda,\nu_1(\rho_m))$ be the space of all $f\in (C^\infty(N_{P_i,\infty}\backslash \SL_2(\C))\otimes (\nL_P\oplus\aL_P)^*\otimes \bar{V}(m))^{K_{P_i,\infty}}$ which additionally satisfy
\[
f(a_{P_i}m_{P_i}g)=\xi_{P_i,\lambda+1}(a_{P_i})\sigma_{P_i}(m_{P_i}^{-1})f(g)\quad\forall g\in \SL_2(\C), 
\quad \forall a_{P_i}\in A_{P_i},\quad\forall m_{P_i}\in M_{P_i}.
\]
Here $\xi_{P_i,\lambda+2}$ and $\sigma_{P_i}$ are the characters which arise from $\xi_\lambda$ and $\sigma$ by conjugating with $g_{P_i}$. If $\nu$ is a finite-dimensional representation of $K$ on a complex vector space $V$, we let $V^{\sigma_{P_i}}$ denote the $\sigma_{P_i}$-isotypical component of for the restriction of the representation $\nu$ to $M_{P_i}$. Then we have the following Lemma.

\begin{lem}\label{Lemma}
For $f\in W_{\bar{\rho}_m}(\sigma,\lambda)$ define $\mu_{\sigma,\lambda}(f)\in \bigoplus_{i=1}^{\kappa(\Gamma)}\left(\left(\aL_P\oplus\nL_P\right)^*\otimes \bar{V}(m)\right)^{\sigma_{P_i}}$ by
\[
\mu_{\sigma,\lambda}(f):=\sum_{i=1}^{\kappa(\Gamma)}\left(\mathcal{I}_{\mathbb{A},P_i}^{-1}\right)^{*}f(1)
=\sum_{i=1}^{\kappa(\Gamma)}f(g_{P_i})
\]
Then $\mu_{\sigma,\lambda}$ defines an isomorphism
\[
\mu_{\sigma,\lambda}:W_{\bar{\rho}_m}(\sigma,\lambda)\cong\bigoplus_{i=1}^{\kappa(\Gamma)}\left(\left(\aL_P\oplus\nL_P\right)^*\otimes \bar{V}(m)\right)^{\sigma_{P_i}}
\]
\end{lem}

\begin{proof}
It is easy to see that each function $\left(\mathcal{I}_{\mathbb{A},P_i}^{-1}\right)^{*}f$ belongs to $\mathcal{E}_{P_i}(\sigma,\lambda,\nu_1(\rho_m))$ and is therefore determined by its value at 1, which moreover belongs to  $\left(\left(\aL_P\oplus\nL_P\right)^*\otimes \bar{V}(m)\right)^{\sigma_{P_i}}$. On the other hand, for $\Phi_{P_i}\in\left(\left(\aL_P\oplus\nL_P\right)^*\otimes \bar{V}(m)\right)^{\sigma_{P_i}}$ we define $\Phi_{P_i;\lambda}\in \mathcal{E}_{P_i}(\sigma,\lambda,\nu_1(\rho_m))$ by
\begin{align}\label{defPhim}
\Phi_{P_i;\lambda}(n_{P_i}a_{P_i}k):=\xi_{P_i;\lambda+1}(a_{P_i})\nu_1(\rho_m)(k^{-1})\Phi_{P_i}.
\end{align}
Let $\mu(\sigma,\lambda)^{-1}(\Phi_{P_i}):=\mathcal{I}_{\mathbb{A}}^{*}(\Phi_{P_i,\lambda})$. Then it immediately follows from the definitions 
that $\mu(\sigma,\lambda)$ and $\mu(\sigma,\lambda)^{-1}$ are inverse to each other.
\end{proof}

We shall from now on identify $T(\mathbb{A})$ with the ring of ideles $\mathbb{A}^*$ by sending $x\in\mathbb{A}^*$ to the diagonal matrix $\diag(x,x^{-1})$. Let $U(\Gamma)_f:=T(\mathbb{A})\cap K(\Gamma)_f$. Let $\sigma=\sigma_k$, $k\in\mathbb{Z}$ and let $\lambda\in\C$. Then we combine $\sigma^{-1}$ and $\lambda$ to a character $\chi_{\infty,\sigma,\lambda}$ of $M A=T_\infty\cong \C^*$ by putting 
\begin{equation} \label{def_character}
\chi_{\infty,\sigma,\lambda}(z):=|z|^{2(\lambda+1)}\left(\frac{\bar{z}}{|z|}\right)^{k}.
\end{equation}

We let $\mathcal{H}(\sigma,\lambda,K(\Gamma)_f)$ denote the set of all Hecke characters $\chi: F^{*}\backslash\mathbb{A}^*\to\C$ which are trivial on $U(\Gamma)_f$ and which satisfy $\chi_{\infty}=\chi_{\infty,\sigma,\lambda}$. If $\left|\cdot\right|_{\mathbb{A}}$ denotes the norm on the adeles, then each $\chi\in\mathcal{H}(\sigma,\lambda,K(\Gamma)_f)$ can be uniquely written as 
\begin{align}\label{formchi}
\chi=\left|\cdot\right|_{\mathbb{A}}^{2(\lambda+1)}\chi_1,
\end{align}
where $\chi_1$ is unitary. For $\chi\in\mathcal{H}(\sigma,\lambda,K(\Gamma)_f)$ with $\chi$ as in \eqref{formchi} we define 
$$
w_0\chi = \chi^{-1} = \left|\cdot\right|_{\mathbb{A}}^{2(-\lambda+1)}\bar{\chi}_1 \in\mathcal{H}(w_0\sigma,-\lambda,K(\Gamma)_f). 
$$ 
Since $T(\mathbb{A})$ normalizes $N(\mathbb{A})$ the group $T(\mathbb{A})$ acts on $W_{\bar{\rho}_m}(\sigma,\lambda)$ by left translations and thus we obtain a decomposition of $W_{\bar{\rho}_m}(\sigma,\lambda)$ into $\chi$-isotypical subspaces:
\begin{align}\label{decHecke}
W_{\bar{\rho}_m}(\sigma,\lambda)=\bigoplus_{\chi\in \mathcal{H}(\sigma,\lambda,K(\Gamma)_f)}W_{\bar{\rho}_m}(\sigma,\lambda)_{\chi}.
\end{align}
Let $t_{P_1},\dots,t_{P_h}\in T(\mathbb{A}_f)$ denote fixed representatives of $T(F)\backslash T(\mathbb{A}_f)/U(\Gamma)_f$. Then for $f\in W_{\bar{\rho}_m}(\sigma,\lambda)$, its projection $f_\chi$ onto $W_{\bar{\rho}_m}(\sigma,\lambda)_{\chi}$ is given by 
\begin{align}\label{projfchi}
f_{\chi}(g)=\frac{1}{\kappa(\Gamma)}\sum_{i=1}^{\kappa(\Gamma)}\bar{\chi}(t_{P_j})f(t_{P_j}g).
\end{align}

Now we use the the notations of the previous sections for the various cohomology groups. We can canonically identify the Lie algebra $\nL_{P_i}$ of $N_{P_i,\infty}$ with $\nL_P$. Then we have canonical embeddings
\[
\bigoplus_{i=1}^{\kappa(\Gamma)}\mathcal{H}^1(\nL_P;\bar{V}(m))_+\hookrightarrow\bigoplus_{i=1}^{\kappa(\Gamma)}(\aL_P\oplus\nL_P)^{*}\otimes \bar{V}(m);\quad
\bigoplus_{i=1}^{\kappa(\Gamma)}\mathcal{H}^1(\nL_P;\bar{V}(m))_-\hookrightarrow\bigoplus_{i=1}^{\kappa(\Gamma)}(\aL_P\oplus\nL_P)^{*}\otimes \bar{V}(m).
\]
Moreover, an easy computation shows that in the present case these embeddings in fact give isomorphisms
\begin{align*}
&\bigoplus_{i=1}^{\kappa(\Gamma)}\mathcal{H}^1(\nL_P;\bar{V}(m))_+\cong\bigoplus_{i=1}^{\kappa(\Gamma)}\left((\aL_P\oplus\nL_P)^{*}\otimes \bar{V}(m)\right)^{\sigma_{m+2}};\\
&\bigoplus_{i=1}^{\kappa(\Gamma)}\mathcal{H}^1(\nL_P;\bar{V}(m))_-\cong\bigoplus_{i=1}^{\kappa(\Gamma)}\left((\aL_P\oplus\nL_P)^{*}\otimes \bar{V}(m)\right)^{\sigma_{-m-2}}.
\end{align*}
Thus together with Lemma \ref{Lemma}, we obtain isomorphisms
\begin{align}\label{muplus}
\mu_{+}(m):W_{\bar{\rho}_m}(\sigma_{m+2},-m/2)\cong \bigoplus_{i=1}^{\kappa(\Gamma)}\mathcal{H}^1(\nL_P;\bar{V}(m))_{+}\cong H^1(\partial\overline{X};\bar{V}(m))_{+}
\end{align}
and
\begin{align}\label{muminus}
\mu_{-}(m):W_{\bar{\rho}_m}(\sigma_{-m-2},m/2)\cong \bigoplus_{i=1}^{\kappa(\Gamma)}H^1(\nL_P;\bar{V}(m))_{-} \cong H^1(\partial\overline{X};\bar{V}(m))_{-}.
\end{align}

We shall denote the operator from $W_{\bar{\rho}_m}(\sigma_{-m-2},m/2)$ to $W_{\bar{\rho}_m}(\sigma_{m+2},-m/2)$ induced by $\underline{\mathbf{C}}(\sigma_{-m-2},m/2)$ and the isomorphisms $\mu_{\pm}(m)$ by $\underline{\mathbf{C}}(\sigma_{-m-2},m/2)$ too. Then it is well-known that for $f\in W_{\bar{\rho}_m}(\sigma_{-m-2},m/2)$ and $g\in G(\mathbb{A})$ one has 
\[
\underline{\mathbf{C}}(\sigma_{-m-2},m/2)f(g)=\int_{N(\mathbb{A})}f(w_0ng)dn,\quad w_0 = \begin{pmatrix} 0 & -1 \\ 1 & 0\end{pmatrix}.
\]
With respect to the decompositions of $W_{\bar{\rho}_m}(\sigma_{-m-2},m/2)$ resp. $W_{\bar{\rho}_m}(\sigma_{m+2},-m/2)$ into Hecke-isotypical subspaces given in \eqref{decHecke}, the operator $\underline{\mathbf{C}}(\sigma_{-m-2,m/2})$ splits as
\[
\underline{\mathbf{C}}(\sigma_{-m-2},m/2)=\bigoplus_{\chi\in \mathcal{H}(\sigma_{-m-2},m/2,K(\Gamma)_f)}C(\chi),
\]
where $C(\chi):W_{\bar{\rho}_m}(\sigma_{-m-2},m/2)_{\chi}\to W_{\bar{\rho}_m}(\sigma_{m+2},-m/2)_{w_0\chi}$. Let $f\in W(\sigma_{-m-2},m/2)_{\chi}$. For our later purposes we can assume that $f=f_\infty\otimes\bigotimes_{v\:\text{finite}}f_v$. Then one has 
\begin{align*}
C(\chi)f=J_\infty(\sigma_{-m-2},m/2)f_\infty\otimes\bigotimes_{v\:\text{finite}}C_v(\chi_v)f_v.
\end{align*}
where the operators $C_v(\chi_v)$ are defined by
\begin{equation}
C_v(\chi_v)f_v = \int_{F_v} f_v\left(w_0\begin{pmatrix} 1&x \\ 0&1 \end{pmatrix}\right) dx
\label{local_intertwine}
\end{equation}
The integral above can be computed explicitly, resulting in the next lemma. Let us set notation for it. Let $\chi_1$ be as in \eqref{formchi}. For $v$ a finite place of $F$ we let $\chi_{1,v}$ be the local component of $\chi_1$ at $v$. Then we say that $\chi_1$ is unramified at $v$ if $\ker(\chi_{1,v})\supseteq 1+\mathcal{O}_v$. Otherwise we say that $\chi_1$ is ramified at $v$. If $\chi_1$ is unramified at $v$, then $\chi_{1,v}(\pi_v)$ does not depend on the choice of $\pi_v$ (this happens in particular if $n_v = 0$). For $v$ a finite place, $\chi_1$ unramified at $v$ and $s\in\C$ the local $L$-factor $L_v(\chi_{1,v},s)$ is defined by 
\begin{equation} \label{def_local_L}
L_v(\chi_{1,v},s):=\frac{1}{1-\chi_{1,v}(\pi_v)|\pi_v|^{-s}}. 
\end{equation}
Then the following Lemma holds. 

\begin{lem}\label{LemIO}
Let $f\in W_{\bar{\rho}_m}(\sigma_{-m-2},m/2)$ and $v$ be a finite place of $F$. Then: 
\begin{itemize}
\item[(i)] If $n_v=0$, one has 
\[
C_v(\chi_v)f_v(1_v)=\frac{L_v(\chi_{1,v},m)}{L_v(\chi_{1,v},m-1)}f_v(\id).
\]

\item[(ii)] If $n_v>0$ and $\chi$ is ramified at $v$, then for $k_v\in K_{max,v}$ one has
\[
C_v(\chi_v)f_v(k_v)=I_v(k_v)f_v(k_v),
\]
where $I_v(k_v)\in |\pi_v|^{-2n_vm}\overline{\mathbb{Z}}$.

\item[(iii)] If $n_v>0$ and $\chi$ is unramified at $v$, then for $k_v\in K_{max,v}$ one has
\[
C_v(\chi_v)f_v(k_v)=\frac{L_v(\chi_{1,v},m)}{L_v(\chi_{1,v},m-1)}I_v(k_v)f_v(k_v),
\]
where $I_v(k_v)\in |\pi_v|^{-2n_vm}\overline{\mathbb{Z}}$.
\end{itemize}
\end{lem}

\begin{proof}
We prove (i) first. For $x\in F_v, |x|_v>1$ we have the Iwasawa decomposition
$$
w_0 \begin{pmatrix} 1 & x \\ 0 & 1 \end{pmatrix} = \begin{pmatrix} x^{-1} & -1 \\ 0 & x\end{pmatrix} \begin{pmatrix} 1 & 0\\x^{-1} & 1 \end{pmatrix}. 
$$
On the other hand for $g = nak$ we have $f_v(g) = \alpha(a)^{s}\chi_{1,v}(a)f(\id)$. Hence we get 
\begin{align*}
I_v &= \int_{|x|_v > 1} |x|_v^{-2s} \chi_{1,v}(x)^{-1} f_v\begin{pmatrix} 1 & 0\\x^{-1} & 1 \end{pmatrix} dx + \int_{\so_v} f_v\begin{pmatrix}  0 & -1 \\ 1 & x \end{pmatrix} dx \\
   &= \left(\int_{|x|_v > 1} |x|_v^{-2s}\chi_{1,v}(x)^{-1} dx + \int_{\so_v}1dx \right) f_v(\id) \\
   &= \left( \sum_{k\ge 1} (1-q^{-1})q^k \cdot \chi_{1,v}(\pi_v)^kq^{-2sk} + 1\right) f_v(\id) \\
   &= \left(1 + \frac{(\chi_{1,v}(\pi_v)q^{2s-1})^{-1}(1 - q^{-1})} {1 - \chi_{1,v}(\pi_v)q^{-2s+1}}\right) f_v(\id) = \frac{1 - \chi_{1,v}(\pi_v)q^{-2s}}{1 - \chi_{1,v}(\pi_v)q^{-2s+1}} \cdot f_v(\id). 
\end{align*}

Now let us prove (ii). It obviously suffices to deal with $k_v = \id$. Recall from the preceding proof that :
\begin{equation}
I_v = \int_{|x|_v > 1} |x|^{-2s} \chi_{1,v}(x) f_v\begin{pmatrix} 1 & 0\\x^{-1} & 1 \end{pmatrix} dx + \int_{\so_v} f_v\begin{pmatrix}  0 & -1 \\ 1 & x \end{pmatrix} dx. 
\label{dec_loc_int}
\end{equation}
The second term is a linear combination of integers (values of $f_v$) with coefficients in $\ZZ[q^{-n}]$ (the measure of a coset on which $f_v$ is constant is equal to $q^{-n}$) and hence lies in $q^{-n}\,\ovl\ZZ$. It remains to deal with the second term. Since $\chi_1$ is ramified at $v$ we have for any $k\in\ZZ$ the equality
$$
\int_{|x|_v = q^k} \chi_{1,v}(x) dx = 0. 
$$
It follows that :
\begin{align*}
\int_{|x|_v > 1} |x|^{-2s} \chi_{1,v}(x) f_v\begin{pmatrix} 1 & 0\\x^{-1} & 1 \end{pmatrix} dx &= \int_{1<|x|_v\le q^n} |x|^{-2s}\chi_{1,v}(x) f_v\begin{pmatrix} 1 & 0\\x^{-1} & 1 \end{pmatrix} dx \\
    &= \sum_{a\in v\so_F/v^n} |a|_v^{2s-2} \chi_{1,v}(a) f_v\begin{pmatrix} 1 & 0 \\ a & 1 \end{pmatrix} |a|^{-1} \int_{1+a^{-1}v^n\so_v} \chi_{1,v}(x) dx. 
\end{align*}
Now there are two possibilities for the integral on the second line: either $\chi_{1,v}$ is trivial on $1+a^{-1}v^n$, in which case the integral equals $q^{-n}|a|^{-1}$, or the integral vanishes. In either case it lies in $q^{-n}\ZZ$, hence the sum lies in $q^{-(2s-1)n}\ovl\ZZ$. 

The proof of (iii) is just a combination of those of (i) and (ii). The decomposition \eqref{dec_loc_int} is still valid, and the estimate for the denominator of the first factor done there is still valid. For the first one we have :
\begin{align*}
\int_{|x|_v < 1} |x|_v^{2-2s}\chi_{1,v}(x) f\begin{pmatrix} 1 & 0 \\ x & 1 \end{pmatrix} dx &= \sum_{a\in\so_F/v^n} f_v\begin{pmatrix} 1 & 0 \\ a & 1 \end{pmatrix} \int_{a+v^n\so_v} |x|_v^{2-2s}\chi_{1,v}(x) dx \\
       &= q^{-n} \sum_{a\not= 0} \begin{pmatrix} 1 & 0 \\ a & 1 \end{pmatrix} \chi_{1,v}(a)|a|_v^{2s-2} + f_v(\id)\int_{|x|_v \le q^{-n}} |x|_v \chi_{1,v}(x) dx 
\end{align*}
(since $\chi_1$ is not ramified at $v$ it is constant on every coset $a+v^n\so_v$). The terms with $a\not=0$ belong to $q_v^{-(2s-1)n}\,\ovl\ZZ$, and the same computation as in the proof of (i) yields that the last term lies in $q^{-2sn} \frac{L_v(\chi_1,2s)}{L_v(\chi_1,2s-1)} \ovl\ZZ$.
\end{proof}

Finally, the term $J_\infty(\sigma_{-m-2},m/2)(f_\infty)$, which is always a ratio of $\Gamma$-functions, can be described explicitly in the present case. 
There is  $\Phi\in \left((\nL_P\oplus\aL_P)^*\otimes V(\rho(m))\right)^{\sigma}$ such that $f_\infty=\Phi_{m/2}$. Moreover, there is an $M_\infty$-equivariant isomorphism
\[
\nu_1(\rho_m)(w_0): \left((\nL_P\oplus\aL_P)^*\otimes V(\rho(m))\right)^{\sigma}\cong \left((\nL_P\oplus\aL_P)^*\otimes V(\rho(m))\right)^{w_0\sigma}.
\]
The representation $\nu_1(\rho_m)$ of $K$ is not irreducible. However, if $\nu_{m+2}$ denotes the representation of $K$ of highest weight $m+2$ in the canonical parametrization, then $\nu_{m+2}$ occurs with multiplicity one in  $\nu_1(\rho_m)$ 
and belong to the $\nu_{m+2}$-isotypical subspace. Thus we have
\begin{align}\label{eqJinfty1}
J_\infty(\sigma_{-m-2},m/2)(\Phi_{m/2})=c_{\nu_{m+2}}(\sigma_{-m-2}:m/2)\cdot\left(\nu_1(\rho_m)(w_0)\Phi\right)_{-m/2},
\end{align}
where $c_{\nu_{m+2}}(\sigma_{-m-2}:m/2)\in\C$ is the value of generalized Harish-Chandra c-function. In the present case, the latter is known explitly. Namely, by \cite[Appendix 2]{Co}, taking the different parametrizations into account, one has
\begin{align}\label{eqJinfty2}
c_{\nu_{m+2}}(\sigma_{-m-2}:m/2)=\frac{1}{\pi}\frac{1}{im+m+2}. 
\end{align}


\section{Estimation of the denominator of the C-matrix}
\label{C_denominator}
\setcounter{equation}{0}
We keep the notation of the preceding section. Our goal here is to prove the following estimate for the denominator of the intertwining matrices. 

\begin{prop}\label{estimate_denominator}
Let $\Gamma$ be a (principal) congruence subgroup of $\Gamma_D$. Then there exists a constant $C_0(\Gamma)$ such that one can estimate
\[
\log |d_{Eis,\mathbf{C}}(\Gamma,\bar{\rho}_m)|\leq C_0(\Gamma)m \log(m)
\]
for all $m\in\mathbb{N}$. 
\end{prop}

Using the maps $\mu_\pm(m)$ from \eqref{muplus} and \eqref{muminus} we obtain distinguished integral lattices $\mu_+^{-1}(m)(H^1(\partial\overline{X};\bar{\Lambda}(m))_{+})$ in the space $W_{\bar{\rho}_m}(\sigma_{m+2},-m/2)$ and $\mu_-^{-1}(m)(H^1(\partial\overline{X};\bar{\Lambda}(m))_{-})$ in $W_{\bar{\rho}_m}(\sigma_{-m-2},m/2)$. More generally, if $R\subset \C$ is a ring with $\mathbb{Z}\subset R$ we will say that $f\in W_{\bar{\rho}_m}(\sigma_{m+2},-m/2)$ is defined over $R$ if it is in the image of $\mu_+^{-1}(m)(H^1(\partial\overline{X};\bar{\Lambda}(m))_{+}\otimes_{\Z}R)$ and we make the corresponding definition for $W_{\bar{\rho}_m}(\sigma_{-m-2},m/2)$. The decomposition of $W_{\bar{\rho}_m}(\sigma_{-m-2},m/2)$ with respect to Hecke characters given in \eqref{decHecke} does not respect the $\Z$-structure on this space just introduced. In other words, if $f\in W_{\bar{\rho}_m}(\sigma_{-m-2},m/2)$ is defined over $\mathbb{Z}$ and if we decompose
\begin{align}\label{Decf}
f=\sum_{\chi\in\mathcal{H}(\sigma_{-m-2},m/2,K(\Gamma)_f)}f_{\chi},
\end{align}
then we cannot expect the $f_\chi$ to be defined over $\Z$. However, we have the following Proposition which controls this defect. 

\begin{prop}\label{propHecke}
There exists an algebraic integer $\alpha\in\overline{\mathbb{Z}}$ which depends on the group $\Gamma$ but not on the representation $\rho(m)$ such that if 
$f\in W_{\bar{\rho}_m}(\sigma_{-m-2},m/2)$ is defined over $\mathbb{Z}$ then $\alpha^mf_{\chi}$ is defined over $\overline{\mathbb{Z}}$ for each $f_\chi$ in the decomposition \eqref{Decf}. 
\end{prop}
\begin{proof}
The character $\chi_{\infty,\sigma_{-m-2},m/2}$ is the character
\begin{align}\label{eqchim}
\chi_{\infty,m+2}:T(\C)\to\C,\:\: \chi_{\infty,m+2}\left(\begin{pmatrix} z &0\\ 0& z^{-1}\end{pmatrix}\right):=z^{m+2}
\end{align}
on $T(\C)$.  Let $\mathcal{H}^1(\nL_P;\bar{\Lambda}(m))_{-}$ denote the integral lattice in corresponding to $H^1(\partial\overline{X};\bar{\Lambda}(m))_{-}$. Without loss of generality, we may assume that $f=\mu_-^{-1}(\Phi)$,  where $\Phi\in H^1(\nL_P;\bar\Lambda(m))_{-}$. We fix $\chi\in\mathcal{H}(\sigma_{-m-2},m/2,K(\Gamma)_f)$. We have
\begin{align*}
\mu_{-}\left((\mu_{-}^{-1}(\Phi))_\chi\right)
=\sum_{i=1}^{\kappa(\Gamma)}(\mu_{-}^{-1}(\Phi))_\chi(g_{P_i}).
\end{align*}
On the other hand, by \eqref{projfchi}, for each $i$ we have
\begin{align}\label{eqmuplusPhi}
(\mu_{-}^{-1}(\Phi))_\chi(g_{P_i})=\frac{1}{\kappa(\Gamma)}\sum_{j=1}^{\kappa(\Gamma)}\chi(t_{P_j})
(\mu_{-}^{-1}(\Phi))(t_{P_j}^{-1}g_{P_i})
\end{align}
For each $i,j$ there exists a unique $l=l(i,j)\in\{1,\dots,h\}$ and a  $g_{\infty}(i,j)\in \SL_2(\C)$ such that 
\begin{align}\label{eqtg}
t_{P_j}^{-1}g_{P_i}=bg_{P_l}g_{\infty}(i,j)k,
\end{align}
where $b\in P(F)$, $k\in K_f(\Gamma)$.  We fix $g_{\infty}(i,j)$ satisfying \eqref{eqtg}. If $g_{P_{l}}\neq g_{P_1}=1$, then,  by definition one has $(\mu_{-}^{-1}(\Phi))(t_{P_j}^{-1}g_{P_i})=0$. One the other hand, if $g_{P_l}=1$, then by definition one has 
\begin{align}\label{eqmuplusPhi2}
(\mu_{-}^{-1}(\Phi))(t_{P_j}^{-1}g_{P_i})=\Phi_{m/2}(g_{\infty}(i,j)),
\end{align}
where $\Phi_{m/2}=\Phi_{P,m/2}\in \mathcal{E}_{P}(\sigma_{-m-2},m/2,\nu_1(\rho_m))$ is as in \eqref{defPhim}. Let $g_\infty(i,j)=p_\infty(i,j)k_\infty(i,j)$, where $p_\infty(i,j)\in P_\infty$, $k_\infty(i,j)\in K_\infty$. Then:
$$
\Phi_{m/2}(g_\infty(i,j)) = \nu_1(\rho_m)(k_\infty(i,j))^{-1}) \Phi_{m/2}(p_{\infty}(i,j)). 
$$
One has $\Phi_{m/2}(p_\infty(i,j))\in(\aL_P\oplus\nL_P)^*\otimes \bar{V}(m))^{\sigma_{-m-2}}$ and $\nu_1(\rho_m)(k_\infty(i,j))^{-1}\Phi_m(p_\infty(i,j))\in(\aL_P\oplus\nL_P)^*\otimes \bar{V}(m))^{\sigma_{-m-2}}$. It is easy to see that this implies $k_\infty(i,j)\in M_\infty$, i.e. $g_\infty\in P_\infty$. Moreover, together with \eqref{eqtg} it follow that $g_\infty\in P(F)$. Thus one can write
\[                                                    
g_{\infty}(i,j)=t(g_{\infty}(i,j))n(g_{\infty}(i,j)),
\]
$t(g_{\infty}(i,j))\in T(F)$, $n(g_{\infty}(i,k))\in N(F)$. We write $t(g_{\infty}(i,j))=\diag(t_{i,j},t_{i,j}^{-1})$ with $t_{i,j}\in F^*$. Then by \eqref{eqchim} and by the definition of $\Phi_{m/2}$ we have $\Phi_{m/2}(g_{\infty}(i,j))=t_{i,j}^{m+2}\Phi$. Thus if $\alpha_{i,j}\in\mathcal{O}_F^*$ is the denominator of $t_{i,j}$, i.e. $\alpha_{i,j} t_{i,j}\in\mathcal{O}_F^*$, it follows that 
\begin{align}\label{eqalphaPhi}
\alpha_{i,j}^{m+2}\Phi_{m/2}(g_\infty(i,j))\in H^1(\nL_P,\bar\Lambda(m))_{-}\otimes_{\mathbb{Z}}\mathcal{O}_F.
\end{align} 
Next, each Hecke character $\chi:F^*\backslash\mathbb{A}^*\to \C$ which is trivial on $U(\Gamma)$ and which satisfies $\chi_{\infty}=\chi_{m+2,\infty}$ is of the form $\tilde{\chi}^{m+2}$, where $\tilde{\chi}:F^*\backslash\mathbb{A}^*\to\C$ is a character which is trivial on $U(\Gamma)$ and satisfies $\tilde{\chi}_{\infty}(\diag(z,z^{-1}))=z$ for $z\in\mathbb{C}^*$. The set of such characters $\tilde{\chi}$ is finite. Thus it follow that there exists a $\beta\in\overline{\mathbb{Z}}$ such that:
\begin{align}\label{eqbeta}
\beta^{m+2}\chi(t_{P_j})\in\overline{\mathbb{Z}}
\end{align} 
for all $j=1,\dots,\kappa(\Gamma)$ and all $\chi\in\mathcal{H}(\sigma_{-m-2},m/2,K(\Gamma)_f)$. Combining \eqref{eqmuplusPhi}, \eqref{eqmuplusPhi2}, \eqref{eqalphaPhi} and \eqref{eqbeta} we get the statement in the proposition. 
\end{proof}

For a unitary Hecke character $\chi_1$ we consider the Hecke L-function
\[
L(\chi_1,s)=\prod_v L_v(\chi_1,v,s)
\]
Recall the for a place $v$ where $\chi$ does not ramify we defined the local factor $L_v$ in \eqref{def_local_L}; we take the convention that $L_v(\chi_1,s):=1$ if $\chi_1$ is ramified at $v$. The infinite product converges absolutely for $\Real(s)>1$ and admits a meromophic continuation to $\C$. For $a\in\ovl\QQ$ we shall denote by $|a|_{\overline{\Q}/\QQ}$ or simply $|a|$ its absolute norm given by
\[
|a|_{\ovl\QQ/\QQ} := \left|B\right|,\quad B = \prod_{b \in\Gal(\overline{\Q}/\Q)\cdot a} b\in\QQ,  
\]
where $\Gal(\overline{\Q}/\Q)$ is the absolute Galois group of $\Q$. The following proposition evaluates the denominator of the quotient of $L$-values appearing in the computation of the intertwining operators in  Lemma \ref{LemIO}. It is essentially contained in the work of Damerell \cite{Damerell1},\cite{Damerell2}: our argument consist in keeping track of norm estimates along the steps in the proof of the main theorem of \cite{Damerell1}. 

\begin{prop} \label{Damerell}
Let $\fri$ be an ideal in $\so_D$. There is $A\in\ZZ_{>0}$ such that for all unitary Hecke characters $\chi$ whose conductor divides $\fri$ and whose infinite part $\chi_\infty(z) = (z/|z|)^n = (\ovl z/z)^{n/2}$ for an even integer $n$, and for any integer $s\in\{0,\ldots,n/2\}$ there is an integer $a'\in\ovl\ZZ$ such that $|a'|\le (n!)^A$ and we have :
\begin{equation}
2^{n-2s+1} a' \frac{L(\chi,s)}{L(\chi,s-1)} \in \pi\ovl\ZZ. 
\label{damerell_plus}
\end{equation}
\end{prop}

\begin{proof}
We first indicate how the arguments used in \cite{Damerell1} yield the following result, which is a more precise version of Theorem 1 in loc. cit. 

\begin{lem}
There is an $\Omega\in\CC^\times$ (depending only on the field $F$) such that the following holds. Let $n$ be an even integer, $\chi$ a unitary Hecke character of $F$ with infinite part $\chi_\infty(z) = (\ovl z/z)^{n/2}$ and $s$ an integer in the range $\{0,\ldots,n/2\}$. Then 
\begin{equation}
\pi^{n/2-s}L(\chi,s)/\Omega^n
\label{Lalg}
\end{equation}
is an algebraic number, whose degree over $\QQ$ is bounded by a constant depending only on $F$ and the conductor $\frf$ of $\chi$ and whose absolute norm is bounded by $C(n!)^A$ for positive integers $C,A$ depending only on $F$. 
\label{L_norm}
\end{lem}

\begin{proof}
{\em In the proof of this lemma all our numbered references are to Damerell's paper \cite{Damerell1}.} Damerell's statement includes only the algebraicity, but his arguments give the full statement above as we shall now explain. Formula (6.2) yields the following expression for the normalized $L$-value occuring above :
\begin{equation}
\pi^{n/2-s}L(\chi,s)/\Omega^n = M \cdot \sum_{i=1}^h \hat \fra_i \chi_f(\hat\fra_i)^{-1} \sum_{\beta\in \fra_i/\frb_i} \chi_f(\beta) \psi^p F_n(\beta\Omega,s,\Omega\frb_i)
\label{Lalg2} \tag{\dag} 
\end{equation}
where: 
\begin{itemize}
\item $M$ is an algebraic number depending on $\frf$ and $F$ ; 
\item $\fra_1,\ldots,\fra_h$ are integral ideals representing the elements in the class-group of $F$ ; 
\item $\frb_i = \frf\fra_i$ ; 
\item $\psi$ is a number depending on $F$ and $p = n/2-s$; 
\item $F_n$ is a particular function which we will analyze below; 
\item $\Omega\in\CC^\times$ is a particular number depending only on $F$. 
\end{itemize}

The next step is Lemma 5.2, which yields an expression of $F_n(\cdot,\cdot,\Lambda)$ in terms of arithmetic invariants of the elliptic curve $\CC/\Lambda$. More precisely, let $\wp$ be the Weierstrass function associated to this elliptic curve, and let $s>1$. Then we have : 
\begin{equation}
\psi^p F_n(z,s,\Lambda) = \sum_{t+u+v = p} \frac{p!}{t!u!v!} h(z)^t (-\varphi)^u (-1)^v K_{q-u}^v(z).
\label{F_n} \tag{$\ast$}
\end{equation}
where :
\begin{itemize}
\item $h(z)$ is a rational fraction (with coefficients in $\QQ$) in $\wp(kz),\wp'(kz),\wp''(kz)$ where $k=1,\ldots,\ell-2$, with $\ell$ the exponent of the finite abelian group $\fra_i/\frb_i$ (see Lemma 4.3) ; 
\item $K_j^i$ is a polynomial (with coefficients in $\QQ$) in $h,\wp(z),\wp'(z),g_2(\Lambda)$ and $g_3(\Lambda)$ and $\varphi$ (Corollary 4.1) ; 
\item $\varphi$ is a constant depending on the curve $\CC/\Lambda$. 
\end{itemize}

The key fact is then that all points at which the various $\wp$ occuring in \eqref{Lalg2} are estimated are of bounded finite order on the elliptic curves: this yields algebraic equations for the relevant values of $\wp$ and its derivative whose degree is bounded (depending on $F$ and $\frf$). The values of $g_2$ and $g_3$ are algebraic of degree depending on the elliptic curve (this is where the choice of $x$ enters, see the remark after Lemma 2.1). It then follows that the values of $\wp''$ are algebraic of bounded degree, as we can see by differentiating in $z$ the equation 
$$
(d\wp/dz)^2 = 4\wp^3 - g_2\wp - g_3
$$
satisfied by $\wp$, which yields
\begin{equation}
d^2\wp/dz^2 = 6\wp^2-g_2/2
\label{wp_seconde} \tag{$\flat$}
\end{equation}
(cf. (3.10),(3.11)). All of this proves that the factors $h(z)$ and $K_{q-u}^v$ in \eqref{F_n} are algebraic of bounded (depending on $F,\frf$) degree. It remains to deal with $\varphi$. Choosing a $\tau\in \so_F$, equation (6.1) yields the expression
$$
\varphi = (\tau\ovl\tau - \tau^2)^{-1}\sum_{\substack{\rho\in\tau b\frb_i/x\frb_i \\ \rho\not= 0}} \wp(\rho/\tau)
$$
which is algebraic of bounded degree. This finishes the proof that all factors of the summands in \eqref{F_n} are algebraic of bounded degree. To finish the proof that the normalized $L$-value itself is so we need only note that since the character $\chi_f$ is of bounded finite order (depending on $\frf$), its values in \eqref{Lalg2} are roots of unity of bounded degree. Thus all terms in \eqref{Lalg2} are algebraic integers of bounded degree. 

\medskip

Now we must bound the absolute norm of the right-hand side in \eqref{Lalg2}. It is obvious from \eqref{Lalg2},\eqref{F_n} and the proof that the degree is bounded that it suffices to prove that the valuation of $K_i^j$, for $2\le i+2\le j\le n$ is bounded by $C(n!)^A$ for some constant $C$. To do this we must return to the arguments of Damerell; in the proof of Lemma 3.3 he shows that for $j\ge 2$ one has 
$$
K_j^0(z) = (-1)^j \frac{d^{j-2}\wp(z)}{dz^{j-2}}. 
$$
From this an easy recursive argument using the identity \eqref{wp_seconde} allows to prove that for $j\ge 2$ $K_j^0(z)$ is a polynomial in $\wp(z),\wp'(z)/2$ and $g_2/12$ of degree less than $2j$ in each variable, with coefficients in $\ZZ$ that are $\ll (2j)!\cdot N^j$ for some integer $N\in\ZZ_{>0}$ ; we will denote this polynomial by $P_j^0 \in \ZZ[X_1,\ldots,X_4]$ (the last variable represents $g_3/4$). 

Damerell proves that for $0\le i<j$ there is a polynomial $P_j^i\in \ZZ[X_1,\ldots,X_4]$ such that $K_j^i(z) = P_j^i(\wp(z),\wp'(z)/2,g_2/12,g_3/4)$. For this he uses the recurrence relation (3.12), which is :
\begin{equation}
K_{j+1}^{i+1}(z) = i \frac{\wp'(z)} 2\cdot \frac 1 j K_j^{i-1}(z) - i\frac{g_2}{12}\cdot \frac 1 j K_{j-1}^{i-1}(z) - \frac 1 j DK_i^j(z)
\label{rec_rel} \tag{$\sharp$}
\end{equation}
where $D$ is a differential operator (in both second variables of $\wp$). It is given explicitely for $\wp,\wp',g_2$ and $g_3$ in the equalities (3.7), which we rewrite here :
\begin{gather*}
Dg_2 = -6g_3, \quad Dg_3 = -\frac 1 3 g_2^2,\\
D\wp = -2\wp^2 - \frac{g_2}3,\quad D\wp' = -3\wp\cdot\wp'. 
\end{gather*} 
Together with \eqref{rec_rel} these finally yield that the degree of $P_j^i$ in each variable is less than $2(i+j)$ and the coefficients are majorized by $2(i+j)!N^{(i+j)}$ for some $N\in\ZZ_{>0}$. It follows that for the values of $z$ occuring in \eqref{Lalg2} we have $|K_{q-u}^v(z)| \ll n!N^n$ at each place, hence the absolute norm is bounded by $(n! N^n)^{[E:\QQ]}$. This finishes the proof of our statement.  
\end{proof}

Bounds for the denominators of special values of $L$-functions are also given by the work of Damerell. We will use the following statement to evaluate denominators of the $L$-part of the intertwining integrals: let $n$ be an even integer, $\chi$ be a unitary Hecke character of $F$, with infinite part
$$
\chi_\infty(z) = (\ovl a/a)^{n/2}
$$
and finite part $\chi_f$ of conductor $\fri$. Then \cite[Theorem 2]{Damerell2} states that for the complex number $\Omega\in\CC^\times$ appearing in Lemma \ref{L_norm} there is an algebraic integer $a\in\ovl\ZZ$, whose absolute norm $|a| = |a|_{\ovl\QQ/\QQ}$ is bounded independently of $s,\chi_\infty$, such that for all integers $s\in\{0,\ldots,n/2\}$ we have
\begin{equation}
2^{n/2-s} a \cdot L(\chi,s) \in \Omega^n/\pi^{n/2-s}\,\ovl\ZZ.
\label{damerell}
\end{equation}
The proposition follows from this and the bound for the abolute norm of normalized $L$-values given in Lemma \ref{L_norm} (the transcendental factors cancel between the numerator and denominator). 
\end{proof}

We can finally put everything together to estimate the denominators of the $\C$-matrix. 

\begin{proof}[Proof of Proposition \ref{estimate_denominator}]
Let $X_\pm$ be the integral vectors in $\nL_P$ of weight $\pm 2$ for $\suL_2$ and $v_\pm$ the vectors of weight $\pm m$ for $\suL_2$ in $\bar V(m)$. Then a cohomology class $\omega_\pm$ in $\mathcal{H}^1(\nL_P ; \bar V(m))_\pm$ is integral if and only if $\int_{c_\pm}\omega\in\ZZ$ where $c_\pm$ is the 1-cycle on $\Gamma_{P_\infty}\bs\HH^3$ with coefficients in $\bar V(m)$ associated to $X_\pm$ and $v_\pm$. Moreover we have 
$$
\int_{c_\pm}\omega_\pm = \mu_\pm(m)^{-1}(\omega)(\id), \quad \int_{c_\pm}\underline{\mathbf{C}}(\omega_\pm) = \mathbf{C}\left(\mu_\mp(m)^{-1}(\omega)\right)(\id). 
$$
Likewise a cohomology class $\omega_\pm\in\mathcal{H}^1(\nL_{P_i} ; \bar V(m))_\pm$ with coefficients in $\bar V(m)$ is integral if and only if the value of a function $f = \mu_{\pm}^{-1}(\omega_\pm)$ at $g_{P_i}$ is integral (up to an at most exponential factor in $m$ coming from the non-integrality of the cycle associated to $\ad(g_{P_i})\cdot X_\pm$ and $\bar\rho_m(g_{P_i})\cdot v_\pm$), and we have the same equivariance property viz. the operators $\underline{\mathbf{C}}$ and $\mathbf{C}$.

The functions $\mu_\pm(m)^{-1}(\omega)(g_{p_i})$ on $G(\Ade)$ are $K_f'$-invariant where $K_f'$ is the compact-open subgroup $\bigcap_{i=1}^h g_{P_i}K_f(\Gamma)g_{P_i}^{-1}$ of $G(\Ade)$. So the statement in the proposition reduces to the following claim: let $K_f'$ be a compact-open subgroup in $\G(\Ade_f)$ and $f\in W_{\bar\rho_m}(\sigma_{m+2},-m/2)$ corresponding to a rational integral cohomology class (the latter being defined as above, with $K_f(\Gamma)$ replaced by $K_f'$). Then we claim that there are $N,C,A\in\ZZ_{>0}$ depending only on $K_f'$ such that we have 
\begin{equation*}
\mathbf{C}(s(m)) f(\id) \in C^{-1}N^{-m}(m!)^A\ZZ f(\id). 
\end{equation*}
To prove we note that it suffices to prove a similar result over $\ovl\ZZ$, namely that for all $f$ as above corresponding to a cohomology class with coefficients in $\bar\Lambda(m)\otimes\ovl\ZZ$ we have 
\begin{equation}
\mathbf{C}(s(m)) f(\id) \in a^{-1}\ZZ f(\id) 
\label{auxiliary}
\end{equation}
for an algebraic integer $a$ with $|a|_{\ovl\QQ/\QQ} \le C(m!)^A$
(indeed, since we know a priori that if the right-hand side is defined over $\QQ$ the proposition follows by thaking the product of Galois conjugates of each side). 

Let us prove \eqref{auxiliary}. First, it follows from Proposition \ref{propHecke} that it suffices to prove it for $f\in W_\chi$. By Lemma \ref{LemIO} and \eqref{eqJinfty1}, \eqref{eqJinfty2} we get that it suffices to prove that 
$$
\frac 1  \pi \cdot \frac{L(\chi,m)}{L(\chi,m-1)} \in b^{-1}\ovl\ZZ
$$
for an $a\in\ovl\ZZ$ with absolute norm $|b|\le(m!)^{A}$. This last statement follows from Proposition \ref{Damerell}. 
\end{proof}


\section{Bounding the torsion from below}
\label{lower_bound}
\setcounter{equation}{0}
In this section we prove the estimate \eqref{esttorsb} (and also \eqref{esttorsh1}, which we actually need to prove the former) from our main Theorem \ref{Mainthrm}. 

\subsection{Torsion in $H^*(\Gamma,\bar\Lambda(m))$}

We first show directly that the order of the group $H^1(\Gamma;\bar{\Lambda}(m))_{tors}$ grows slower in $m$ than our leading term. Since we work with a split algebraic group the proof of this is simpler than that of the corresponding statement in \cite{MaM}. 

\begin{lem}\label{PropH1}
Let $\Gamma$ be a congruence subgroup of $\Gamma_D$. Then  
\[
\log|H^1(\Gamma;\bar{\Lambda}(m))_{tors}| = O(m\log m),
\]
as $m\to\infty$. 
\end{lem}

\begin{proof}
One has $H^1(\Gamma;\bar{\Lambda}(m))_{tors}\cong H_0(\Gamma;\bar{\Lambda}(m))_{tors}$ by the universal coefficient theorem. Let $\Lambda^0(m)$ denote the submodule of $\Lambda(m)$ generated by all $(\rho_m(\gamma)-\Id)v$, where $v\in\Lambda(m)$ and $\gamma\in\Gamma$. Then $H_0(\Lambda(m))=\Lambda(m)/\Lambda^0(m)$. There exists an $a\in\mathbb{N}$ such that for $n_a:=\begin{pmatrix}1&a\\ 0&1\end{pmatrix}$ and $\bar{n}_a:=\begin{pmatrix}1&0\\ a&1\end{pmatrix}$ one has $n_a,\:\bar{n}_a\in\Gamma$. If we let $X$, $Y$ denote the standard basis of $\C^2$, then $X^m,\:X^{m-1}Y,\:\dots,Y^m$ is a basis of $\Lambda(m)$ and in this basis, $\rho_m(n_a)-\Id$ is represented by an upper triangular nilpotent matrix. For $j> i$, the entry in the $i$-th row, $j$-th column of this matrix is given by $a^{j-i}\begin{pmatrix}j-1\\ j-i\end{pmatrix}$. Thus it follows inductively that $(l+1)aX^{m-l}Y^l\in \Lambda^0(m) $ for $0\leq l<m$. On the other hand, one has $(\rho_m(\bar{n}_a)-\Id)XY^{m-1}=aY^m$. Thus one has $|H_0(\Lambda(m))|\leq a^{m+1}m!$. i.e. $\log|H_0(\Lambda(m))|=O(m\log m)$ as $m\to\infty$. For $\check{\Lambda}(m)$ one can argue similarly.  
\end{proof}

The main result we prove here is the exponential growth of the torsion subgroup of $H^2(\Gamma,\bar\Lambda(m))$; in the remainder of this subsection we will show how all the work done in sections 4--6 implies the following result. 

\begin{prop}\label{aux_torsgrowth}
We have 
\begin{equation}
\liminf_{m\to+\infty} \frac {\log|H^2(\Gamma(\aL),\bar\Lambda(m))_{tors}|}{m^2} \ge \frac{\vol(X_\aL)}\pi \left( 1 - \frac{N(\aL_0) }{N(\aL)}\right)
\end{equation}
\end{prop}

\begin{proof}
We first remark that it follows easily from the interpretation of $H^1(\partial\overline{X};\bar{\Lambda}(m))_{\pm}$ in terms of closed $E_{\bar\rho_m}$-valued holomorphic respectively antiholomorphic forms on the boundary $\partial\overline{X}$ that 
\begin{align}\label{estb}
[H^1(\partial\overline{X};\bar{\Lambda}(m))_{free}:H^1(\partial\overline{X};\bar{\Lambda}(m))_-\oplus 
H^1(\partial\overline{X};\bar{\Lambda}(m))_+]= O(m\log m),
\end{align}
as $m\to\infty$. 

Let $\mathcal{A}$ and  $C_1(\Gamma)$ be as in the end of section \ref{tors_cong}. Let $\aL$ be a non-zero ideal of $\mathcal{O}_D$ with $N(\aL)>C_1(\Gamma)$ and let $\aL_0\in\mathcal{A}$ such that $n_{l,\Gamma(\aL)}=n_{l,\Gamma(\aL_0)}$ for each $l=1,\dots,d_F$. For brevity we shall use the following notation in the remaining computations:
$$
R^i(X,\bar{\mathcal{L}}(m)) = \vol_{\mathcal B^i(\Gamma,\bar\rho_m)} \left( H^i(\Gamma,\bar\Lambda(m))\right). 
$$
With this notation \eqref{eqBV} becomes
\begin{equation}\label{eqBVredux}
\begin{split}
\log\tau_{Eis}(X; E_{\bar\rho_m}) &= \log R^2(X_\aL,\bar{\mathcal{L}}(m)) + \log|H^1(\Gamma(\aL),\bar\Lambda(m))_{tors}| \\ 
           &\quad\quad - \log R^1(X_\aL,\bar{\mathcal{L}}(m)) - \log|H^2(\Gamma(\aL),\bar\Lambda(m))_{\tors}|. 
\end{split}
\end{equation}

By Lemmas \ref{Lemreg2} and \ref{PropH1} we have that 
\begin{equation}
\left|\log\left(|H^1(\Gamma(\aL),\bar\Lambda(m))_{tors}|\right)\right|, \: |\log R^2(X_\aL,\bar{\mathcal{L}}(m))| \ll m\log(m). 
\label{negligible}
\end{equation}
On the other hand, by Lemma \ref{Lemreg1} together with Proposition \ref{estimate_denominator} and \eqref{estb} we have that
$$
\liminf_{m\to+\infty}\frac{\log R^1(X_\aL,\bar{\mathcal{L}}(m))}{m^2} \le \liminf_{m\to+\infty}\frac{\log|H^2(\Gamma(\aL),\Lambda(m))_{tors}|}{m^2}
$$
From \eqref{negligible} we get that in the expression \eqref{eqBVredux} for the Reidemeister torsion $\tau_{Eis}(X_\aL;E_{\bar\rho_m})$ all terms but for $\log|H^2|$ and $\log R^1$ are $O(m\log(m))$ and using the preceding inequality we get that
\begin{equation}
\liminf_{m\to+\infty} \left( 2\frac{\log|H^2(\Gamma(\aL),\Lambda(m))_{tors}|}{m^2} \right) \ge \liminf_{m\to+\infty}  \left(-\frac{\log\tau_{Eis}(X_{\aL}; E_{\bar\rho_m})}{m^2}\right). 
\label{limfin1}
\end{equation}

On the other hand, we also get from \eqref{estb} and Lemma \ref{Lemreg1} that
$$ 
\liminf_{m\to+\infty} \left(\frac{-\log R^1(X_{\aL_0},\bar{\mathcal{L}}(m))}{m^2}\right) \le 0 
$$
from which and \eqref{negligible} (used for $\aL_0$ instead of $\aL$) it follows that
\begin{equation}
\liminf_{m\to+\infty} \left(\frac{\log\tau_{Eis}(X_{\aL_0}; E_{\bar\rho_m})}{m^2}\right) \le 0. 
\label{limfin2}
\end{equation}

Putting together \eqref{limfin1} and \eqref{limfin2} we get that
\begin{multline*}
\liminf_{m\to+\infty} \left(2\frac{[\Gamma_D:\Gamma(\aL_0)]}{|\so_D^*|\cdot N(\aL_0)} \cdot \frac{\log|H^2(\Gamma(\aL),\Lambda(m))_{tors}|}{m^2} \right) \ge \\ 
          \liminf_{m\to+\infty} \left( \frac{-[\Gamma_D:\Gamma(\aL_0)]}{|\so_D^*|\cdot N(\aL_0)} \cdot \frac{\log\tau_{Eis}(X_{\aL}; E_{\bar\rho_m})}{m^2} + \frac{[\Gamma_D:\Gamma(\aL)]}{|\so_D^*|\cdot N(\aL)} \cdot \frac{\log\tau_{Eis}(X_{\aL_0}; E_{\bar\rho_m})}{m^2} \right). 
\end{multline*}
Finally, the right-hand side above converges to 
\begin{multline*}
2\cdot\frac{[\Gamma_D:\Gamma(\aL_0)]\cdot[\Gamma_D:\Gamma(\aL)]}{|\mathcal{O}_D^*|\pi} \left(\frac{1}{N(\aL_0)}-\frac{1}{N(\aL)}\right)\vol(\Gamma_D\backslash\mathbb{H}^3) \\
   = 2\frac{[\Gamma_D : \Gamma(\aL_0)]}{|\mathcal{O}_D^*|N(\aL_0)} \cdot \frac{\vol(X_\aL)}{\pi}\cdot\left( 1 - \frac {N(\aL_0)}{N(\aL)}\right)
\end{multline*}
by Proposition \ref{asRT} (which we are allowed to use for the representation $\rho_m\oplus\rho_m$ instead of $\rho_m$ since the analytic or Reidemeister torsion of the former is the square of that of the latter). This finishes the proof of Proposition \ref{aux_torsgrowth}. 
\end{proof}


\subsection{Independance from the lattice}

Here we prove that $\log|H^2(\Gamma, \Lambda_m)_{tors}|$ does not depend on the choice of lattices $\Lambda_m\subset V(m)$ up to an error term of size $m\log(m)$. 

\begin{prop}\label{comp_homtors}
Let $\Gamma$ be a finite-index subgroup of the Bianchi group $\Gamma_D$. There is a constant $C$ depending only on $\Gamma$ such that for any $m\ge 1$ and any two $\Gamma$-invariant lattices $\Lambda_1,\Lambda_2$ in $V(m)$ we have 
$$
\left|\log\left(\frac{|H^2(\Gamma,\Lambda_1)_{tors}|}{|H^2(\Gamma,\Lambda_2)_{tors}|} \right) \right| \le Cm\log(m). 
$$
\end{prop}

We will deduce the proposition from the two next lemmas. 

\begin{lem}\label{comp_reg}
Let $\Gamma$ be a subgroup of a Bianchi group, $\rho$ be a representation of $\SL_2(\CC)$ on a vector space $V$ and $\Lambda,\Lambda'$ two $\rho(\Gamma)$-invariant lattices in $V$ such that $M\cdot\Lambda\subset\Lambda'\subset\Lambda$ for some integer $M\in\ZZ_{>0}$. Let $\mathcal L,\mathcal L'$ be the local systems on $X = \Gamma\bs\HH^3$ induced by $\Lambda,\Lambda'$ and $E_\rho$ the Euclidean bundle on $X$ induced by $\rho$. Then we have
$$
1 \le \frac{R^1(X,\mathcal L')}{R^1(X,\mathcal L)} \le M^{\dim H^1(X,E_\rho)}.  
$$
\end{lem}

\begin{proof}
Let $h = \dim H^1(X,E_\rho)$ and let $c_1,\ldots,c_h\in Z^1(X,\mathcal L)$ such that the cohomology classes $[c_1],\ldots,[c_h]$ generate the free part of $H^1(X,\mathcal L)$. Then each $M\cdot c_i$ belongs to $Z^1(X,\mathcal L')$ and together the $M\cdot[c_i]$ generate a finite-index subgroup of $H^1(X,\mathcal L')$. Thus we get 
$$
M\cdot H^1(X,\mathcal L) \subset H^1(X,\mathcal L')
$$
and the inequality 
$$
R^1(X,\mathcal L') \le [H^1(X,\mathcal L) : M\cdot H^1(X,\mathcal L)] R^1(X,\mathcal L) = M^h R^1(X,\mathcal L)
$$
follows immediately. 
\end{proof}

\begin{lem}\label{comp_lattices}
There is a constant $c\in\ZZ_{>0}$ depending on $\Gamma$ such that if $\Lambda_1,\Lambda_2$ are two $\Gamma$-invariant lattices in $V(m)$ then there exists $a\in\QQ$ such that
$$
a\Lambda_1 \subset \Lambda_2 \subset a(m!)^{-c}\Lambda_1.
$$
\end{lem}

\begin{proof}
Let $\langle\cdot,\cdot\rangle$ be the pairing on $V(m) = \Symm^m\CC^2$ induced  by the determinant on $\CC^2\times\CC^2$ (which is hence nondegenerate and $\Gamma$-invariant) and let $\Lambda_1'$ be the $\langle\cdot,\cdot\rangle$-dual lattice of $\Lambda_1$ in $V(m)$, that is
$$
\Lambda_1' = \{v\in V(m) :\: \forall u\in\Lambda_1,\, \langle u,v\rangle\in\ZZ\}. 
$$
Then 
\begin{equation}
m!\Lambda_1' \subset \Lambda_1 \subset (m!)^{-1}\Lambda_1' 
\label{dual_index}
\end{equation}
as follows from the expression of $\langle,\rangle$ in coordinates (see for example \cite[2.4]{Berger}). Now let $u$ be a primitive vector in $\Lambda_2$ which is a vector of maximal weight for the standard parabolic subgroup of $\SL_2(\CC)$ in $V(m)$ (i.e. a rational multiple of $X^m$); there exists an $a\in\QQ$ such that $au$ is a primitive vector in $\Lambda_1'$. Then $\Lambda_3 := \langle\Gamma\cdot au\rangle \subset \Lambda_1'$: indeed, for any $v\in\Lambda_1$ and $\gamma\in\Gamma$ we have 
$$
\langle v,\gamma\cdot au\rangle  = \langle\gamma^{-1}\cdot v, au\rangle \in\ZZ. 
$$
From this and \eqref{dual_index} we get that $\Lambda_3\subset (m!)^{-1}\Lambda_1$. By arguments similar to those used in the proof of Lemma \ref{PropH1} (which we will detail after we explain how to conclude the proof from there) , we also have that 
\begin{equation} \label{contain}
a\Lambda_2\subset N^{-m}(m!)^{-2}\Lambda_3
\end{equation}
for some $N\in\ZZ_{>0}$ depending on $\Gamma$. It finally follows that
$$
a\Lambda_2\subset N^{-m}(m!)^{-3}\Lambda_1
$$
which proves half the lemma ; the second half also follows by a completely symmetric argument. 

Let us explain how \eqref{contain} is proved. For ease of notation we will suppose that $au = X^m$. We put $\Lambda_0 = \Lambda(m) = \so_D X^m\oplus\so_D X^{m-1}Y\oplus\ldots\oplus \so_D Y^m$. Let $z\in\ZZ$ such that $n_z\in\Gamma$ and $k\in\{1,\ldots,m\}$ ; suppose that $bX^{m-k}Y^k\in\Lambda_2$ for some $b\in F$. Then we get that 
$$
\Lambda_2 \ni n_z^k\cdot bX^{m-k}Y^k = bk!z^kX^m 
$$ 
so that $bk!z^m\in\so_D$. Thus we get that $\Lambda_2\subset (z^m m!)^{-1}\Lambda_0$. On the other hand the proof of Lemma \ref{PropH1} yields that $[\Lambda_0:\Lambda_3] \le z^mm!$, hence $\Lambda_2\subset (z^m m!)^{-2}\Lambda_3$. 
\end{proof}

\begin{proof}[Proof of Proposition \ref{comp_homtors}]
Let $\mathcal L_i$ be the local system on $X$ induced by the lattice $\Lambda_i$. We have by \eqref{eqBVredux} that 
$$
\frac{R^2(X,\mathcal L_1)\cdot |H^1(\Gamma,\Lambda_1)_{tors}|}{R^1(X,\mathcal L_1)\cdot |H^2(\Gamma,\Lambda_1)_{tors}|} = \tau_{Eis}(X, \rho_m) = \frac{R^2(X,\mathcal L_2)\cdot |H^1(\Gamma,\Lambda_2)_{tors}|}{R^1(X,\mathcal L_2)\cdot |H^2(\Gamma,\Lambda_2)_{tors}|}. 
$$
By Lemmas \ref{Lemreg2} and \ref{PropH1} we have that $H^1,R^2$ are $\ll m\log(m)$ for whichever lattice, and by Lemmas \ref{comp_reg} and \ref{comp_lattices} we get that 
$$
\left|\log\left(\frac{R^1(X,\mathcal L_1)}{R^1(X,\mathcal L_2)}\right) \right| \ll m\log(m)
$$
and we can thus conclude that the remaining terms $\log(H^2(\Gamma,\Lambda_i))$ in the Reidemeister torsion differ by at most $Cm\log(m)$ for some $C>0$ depending on $\Gamma$. 
\end{proof}


\subsection{Conclusion}

Let $\bar\Lambda'(m)$ be the lattice $\Lambda(m)\oplus\Lambda(m)$ in $V(m)\oplus V(m)$. By Propositions \ref{comp_homtors} and \ref{aux_torsgrowth} we get that 
$$
\liminf\frac{\log|H^2(\Gamma(\aL), \bar\Lambda'(m))_{tors}|}{m^2} \ge \frac{\vol(X_{\aL})}{\pi}\left(1 - \frac{N(\aL_0)}{N(\aL)}\right). 
$$
Since $H^2(\Gamma(\aL), \bar\Lambda'(m)) \cong H^2(\Gamma(\aL), \Lambda(m))^2$ we get \eqref{esttorsb}. The estimate \eqref{esttorsh1} is proven exactly as in Lemma \ref{PropH1}.


\section{Bounding the torsion from above}
\label{upper_bound}
\setcounter{equation}{0}
In this section we proof equation \eqref{esttorsab} from Theorem \ref{Mainthrm} : we give the proof for $\bar\Lambda(m)$-coefficients, the case of $\Lambda(m)$ follows immediately by Proposition \ref{comp_homtors}. The main ingredient is the following lemma, usually attributed to O. Gabber and C. Soul\'e (we note that it is also an important tool in V. Emery's a priori bound for the torsion in the homology of certain arithmetic lattices, see \cite{Eme}). We refer the reader to \cite[Lemma 3.2]{Sa} for a proof. 

\begin{lem}\label{LemmGS}
Let $A:=\mathbb{Z}^a$ with standard basis 
$(e_i)_{i=1,\dots,a}$ and let  $B:=\mathbb{Z}^b$. Equip 
$B\otimes_{\Z}\mathbb{R}$ with 
the Euclidan norm $\left\|\cdot\right\|$. Let 
$\phi:A\to B$ be $\Z$-linear and
assume that there exists $\alpha\in\R$ such that $\left\|\phi(e_i)\right\|\leq \alpha$ for each $i=1,\dots,a$. Then 
one has
\[
|coker(\phi)_{tors}|\leq\alpha^{min(a,b)}
\]
\end{lem}

Next, we have the following elementary Lemma.
\begin{lem}\label{Lemnorm}
Let $\{v_j:=e_1^{m-i}e_2^j\colon i=0,\dots,m\}$, 
denote the standard integral basis of the lattice $\Lambda_m\subset V(m)$. 
Equip $V(m)$ with the inner product such that the $v_i$ form an 
orthonormal basis and let $\left\|\cdot\right\|_{\End(V_m)}$  
denote the corresponding norm on $\End(V_m)$. Then 
for each $\gamma\in\Gamma_D$ one has
\begin{align*}
\left\|\rho_m(\gamma)\right\|_{\End(V_m)}\leq\left\|\rho_1(\gamma)\right\|_{\End(V_1)}^{m}
\end{align*}
\end{lem}
\begin{proof}
This follows immediately from the definitions. 
\end{proof}

Now we can estimate the torsion from above as follows.

\begin{prop}\label{propestab}
Let $\Gamma$ be a finite index, torsion-free subgroup of $\Gamma_D$. Then 
there exists a consant $c_\Gamma$ such that one can 
estimate
\[
\log |H^2(\Gamma,\bar{\Lambda}(m))_{tors}|\leq c_\Gamma m^2
\]
for each $m\in\mathbb{N}$. 
\end{prop}

\begin{proof}
Let $X:=\Gamma\backslash\mathbb{H}^3$. Let $\mathcal{K}$ 
be a smooth triangulation of $X$ and let $\tilde{\mathcal{K}}$ 
denote its lift to a smooth triangulation of $\mathbb{H}^3$. 
For each $q$ let
$C_q(\mathcal{K}):=\{\sigma_{1,q},\dots,\sigma_{N(\Gamma,q),q}\}$ denote the simplicial $q$-chains 
of $\mathcal{K}$, where $N(\Gamma,q)\in\mathbb{N}$ depends on $\mathcal{K}$. 
Let $C_q(\tilde{\mathcal{K}})$ 
denote the simplicial $q$-chains of $\tilde{\mathcal{K}}$ and let 
$\tilde{\partial}_q:C_{q}(\tilde{\mathcal{K}})\to C_{q-1}(\tilde{\mathcal{K}})$ be the 
corresponding boundary operator.  
For each $\sigma_{i,q}$  we fix a $\tilde{\sigma}_{i,q}\in\tilde{\mathcal{K}}$ such 
that $\pi_*(\tilde{\sigma}_{i,q})=\sigma_{i,q}$, where $\pi:\tilde{X}\to X$ is 
the covering map. 
The group 
$\Gamma$ acts on $C_q(\mathcal{K})$ and we denote the 
corresponding action simply by $\cdot$. 
For each $\sigma_{i,q}$ there exist elements
$\gamma_{k,q-1}\in\Gamma$, $k=1,\dots N(q-1,\Gamma)$, such that
\[
\tilde{\partial}_q(\tilde{\sigma}_{i,q})=\sum_{k=1}^{N(q-1,\Gamma)}\gamma_{k,q-1}\cdot\tilde{\sigma}_{k,q-1}.
\]

Let $C_q(\mathcal{K};\Lambda_m):=C_q(\mathcal{K})\otimes_{\mathbb{Z}[\Gamma]}\Lambda_m$. 
Then the homology groups $H_*(\Gamma;\Lambda_m)$ are isomorphic to the homology 
groups $H_*(C_*(\mathcal{K};\Lambda_m))$ of the complex
\[
(C_*(\mathcal{K};\Lambda_m),\partial_{*;\rho_m}):=(C_*(\mathcal{K})\otimes_{\mathbb{Z}[\Gamma]}\Lambda_m,\tilde{\partial}_{*}\otimes\Id). 
\] 
Let $\{v_0,\dots,v_m\}$ denote the standard integral 
basis of $\Lambda_m$ as in Lemma \ref{Lemnorm}.
Then an integral basis of $C_q(\mathcal{K};\Lambda_m)$ is given by 
\[
B_q(\mathcal{K};\Lambda_m):=\{\tilde{\sigma}_{i,q}\otimes v_j\colon i=1,\dots,N(\Gamma,q)\colon j=0,\dots,m\}.
\] 
We equip $C_q(\mathcal{K};\Lambda_m)\otimes_\Z\R$ with 
the inner product for which $B_q(\mathcal{K};\Lambda_m)$ is an orthonormal 
basis and denote the corresponding 
norm by $\left\|\cdot\right\|_{C_q(\mathcal{K};V(m))}$. 
Then 
we have
\[
\partial_{q;\rho_m}(\tilde{\sigma}_{i,q}\otimes v_j)=\sum_{k=1}^{N(q-1,\Gamma)}\tilde{\sigma}_{k,q-1}\otimes(\rho_m(\gamma_{k,q-1}^{-1})v_j)
\]
and thus by the definition of the norms we have
\begin{align}\label{est}
&\left\|\partial_{q;\rho_m}(\tilde{\sigma}_{i,q}\otimes v_j)\right\|_{\mathcal{C}_{q-1}(\mathcal{K};V(m))}\leq N(q-1,\Gamma)
\max_{k=1,\dots,N(q-1,\Gamma)}\left\|\rho_m(\gamma_{k,q-1}^{-1})\right\|_{\End(V_m)}
\nonumber\\=&N(q-1,\Gamma)\left(\max_{k=1,\dots,N(q-1,\Gamma)}\left\|\rho_1(\gamma_{k,q-1}^{-1})\right\|_{\End(V_1)}\right)^m,
\end{align}
where the last step follows from Lemma \ref{Lemnorm}.
We put 
\[
c_0(\Gamma):=\max_q\max_{k=1,\dots,N(q-1,\Gamma)}\left\|\rho_1(\gamma_{k,q-1}^{-1})\right\|_{\End(V_1)}.
\]
Then $c_0(\Gamma)$ depends on $\Gamma$ and the triangulation $\mathcal{K}$, but not on the 
local system $\Lambda_m$. If  
we apply Lemma \ref{LemmGS} with 
$A:=C_q(\mathcal{K};\Lambda_m)$, $B:=C_{q-1}(\mathcal{K};\Lambda_m)$, $\phi:=\partial_{q;\rho_m}$ and  $\alpha:=N(\Gamma,q-1)c_0(\Gamma)^{m}$, 
then, using that $\rk_\Z A=(m+1)N(q,\Gamma)$, $\rk_z B=(m +1)N(q-1,\Gamma)$ we obtain from \eqref{est} that 
\[
\left|\left(coker (\partial_{q;\rho_m})\right)_{tors}\right|\leq \left(N(\Gamma,q-1)c_0(\Gamma)^{m}\right)^{(m+1)\min\{N(q,\Gamma),N(q-1,\Gamma)\}}
\]
For $\check{\Lambda}_m$ one argues in the same way.
Thus the proposition follows by applying the universal coefficient theorem.  
\end{proof}

\begin{bmrk}
Using the $KAK$-decomposition, it 
should be possible to generalize Lemma \ref{Lemnorm} 
and thus the proof of Proposition \ref{propestab} 
to arithmetic subgroups $\Gamma$ of arbitrary connected semisimple Liegroups $G$ defined 
over $\Q$ which satisfy $\delta(G)=1$. 
For suitable rays $\rho_\lambda(m)$ of $\Q$-rational representations of $G$ 
of highest weight $m\lambda$ with $\Gamma$-invariant 
integral lattices $\Lambda(\rho_\lambda(m))$, this should give an upper 
bound of the corresponding sizes of all twisted 
cohomological torsion subgroups $H^*_{tors}(\Gamma,\Lambda(\rho_\lambda(m)))$ by $C(\Gamma)m\dim\rho_\lambda(m)$. Such a bound  
can be regarded as complementary to the lower bound 
obtained in the compact case in \cite{MPcoho}.
\end{bmrk}
Let $\aL_0\in\mathcal{A}$ be as in the previous section.
If we apply Proposition \ref{propestab} for the 
group $\Gamma(\aL_0)$ instead of the group $\Gamma(\aL)$ and 
also use Proposition \ref {asRT}, we can improve 
the constant in the upper bound of the size of $m^{-2}\log|H^2(\Gamma(\aL),\bar{\Lambda}(m))_{tors}|$ and 
thus prove \eqref{esttorsab}. 
Namely, arguing similar as in the proof 
of \eqref{esttorsb} given in the previous section, we obtain
\begin{align*}
&\lim\sup_{m\to\infty}m^{-2}\frac{[\Gamma_D:\Gamma(\aL_0)]}{\#(\mathcal{O}_D^*) N(\aL_0)}\log|H^2_{tors}(\Gamma(\aL),\bar{\Lambda}(m))|-\frac{[\Gamma_D:\Gamma(\aL)]}{\#(\mathcal{O}_D^*) N(\aL)}c(\Gamma_0)\\
\leq &\lim\sup_{m\to\infty}m^{-2}\left(-\frac{[\Gamma_D:\Gamma(\aL_0)]}{\#(\mathcal{O}_D^*)N(\aL_0)}\log\tau_{Eis}(X_{\aL};E_{\bar{\rho}(m)})+
\frac{[\Gamma_D:\Gamma(\aL)]}{\#(\mathcal{O}_D^*) N(\aL)}\log\tau_{Eis}(X_{\aL_0};E_{\bar{\rho}(m)})\right)
\\ &+\frac{[\Gamma_D:\Gamma(\aL_0)]}{\#(\mathcal{O}_D^*) N(\aL)}c(\Gamma(\aL_0)).
\end{align*}
Invoking Proposition \ref{asRT}, equation \eqref{esttorsab} follows.


\bibliographystyle{alpha}
\bibliography{bib}

\begin{thebibliography}{{Sau}14}

\bibitem[Ber08]{Berger}
Tobias Berger.
\newblock Denominators of {E}isenstein cohomology classes for {${\rm GL}_2$}
  over imaginary quadratic fields.
\newblock {\em Manuscripta Math.}, 125(4):427--470, 2008.

\bibitem[BV13]{BV}
Nicolas Bergeron and Akshay Venkatesh.
\newblock The asymptotic growth of torsion homology for arithmetic groups.
\newblock {\em J. Inst. Math. Jussieu}, 12(2):391--447, 2013.

\bibitem[Coh74]{Co}
Leslie Cohn.
\newblock {\em Analytic theory of the {H}arish-{C}handra {$C$}-function}.
\newblock Lecture Notes in Mathematics, Vol. 429. Springer-Verlag, Berlin-New
  York, 1974.

\bibitem[CV12]{CV}
F.~{Calegari} and A.~{Venkatesh}.
\newblock {A torsion Jacquet--Langlands correspondence}.
\newblock {\em ArXiv e-prints}, December 2012.

\bibitem[Dam70]{Damerell1}
R.~M. Damerell.
\newblock {$L$}-functions of elliptic curves with complex multiplication. {I}.
\newblock {\em Acta Arith.}, 17:287--301, 1970.

\bibitem[Dam71]{Damerell2}
R.~M. Damerell.
\newblock {$L$}-functions of elliptic curves with complex multiplication. {II}.
\newblock {\em Acta Arith.}, 19:311--317, 1971.

\bibitem[DH99]{DH}
Anton Deitmar and Werner Hoffman.
\newblock Spectral estimates for towers of noncompact quotients.
\newblock {\em Canad. J. Math.}, 51(2):266--293, 1999.

\bibitem[EGM98]{EGM}
J.~Elstrodt, F.~Grunewald, and J.~Mennicke.
\newblock {\em Groups acting on hyperbolic space}.
\newblock Springer Monographs in Mathematics. Springer-Verlag, Berlin, 1998.
\newblock Harmonic analysis and number theory.

\bibitem[Eme14]{Eme}
Vincent Emery.
\newblock Torsion homology of arithmetic lattices and {$K_2$} of imaginary
  fields.
\newblock {\em Math. Z.}, 277(3-4):1155--1164, 2014.

\bibitem[FGT10]{FGT}
Tobias Finis, Fritz Grunewald, and Paulo Tirao.
\newblock The cohomology of lattices in {${\rm SL}(2,\Bbb C)$}.
\newblock {\em Experiment. Math.}, 19(1):29--63, 2010.

\bibitem[Har75]{Ha}
G.~Harder.
\newblock On the cohomology of discrete arithmetically defined groups.
\newblock In {\em Discrete subgroups of {L}ie groups and applications to moduli
  ({I}nternat. {C}olloq., {B}ombay, 1973)}, pages 129--160. Oxford Univ. Press,
  Bombay, 1975.

\bibitem[Kos61]{Kostant}
Bertram Kostant.
\newblock Lie algebra cohomology and the generalized {B}orel-{W}eil theorem.
\newblock {\em Ann. of Math. (2)}, 74:329--387, 1961.

\bibitem[Les13]{Lesch}
Matthias Lesch.
\newblock A gluing formula for the analytic torsion on singular spaces.
\newblock {\em Anal. PDE}, 6(1):221--256, 2013.

\bibitem[MM13]{MaM}
Simon Marshall and Werner M{\"u}ller.
\newblock On the torsion in the cohomology of arithmetic hyperbolic
  3-manifolds.
\newblock {\em Duke Math. J.}, 162(5):863--888, 2013.

\bibitem[MP12]{MPtors1}
Werner M{\"u}ller and Jonathan Pfaff.
\newblock Analytic torsion of complete hyperbolic manifolds of finite volume.
\newblock {\em J. Funct. Anal.}, 263(9):2615--2675, 2012.

\bibitem[MP14a]{MPtors}
Werner M{\"u}ller and Jonathan Pfaff.
\newblock The analytic torsion and its asymptotic behaviour for sequences of
  hyperbolic manifolds of finite volume.
\newblock {\em J. Funct. Anal.}, 267(8):2731--2786, 2014.

\bibitem[MP14b]{MPcoho}
Werner M{\"u}ller and Jonathan Pfaff.
\newblock On the growth of torsion in the cohomology of arithmetic groups.
\newblock {\em Math. Ann.}, 359(1-2):537--555, 2014.

\bibitem[{Pfa}13]{Pf}
J.~{Pfaff}.
\newblock {A gluing formula for the analytic torsion on hyperbolic manifolds
  with cusps}.
\newblock {\em To appear in JIMJ, ArXiv e-prints}, December 2013.

\bibitem[{Rai}13]{congruence}
J.~{Raimbault}.
\newblock {Analytic, Reidemeister and homological torsion for congruence
  three--manifolds}.
\newblock {\em ArXiv e-prints}, July 2013.

\bibitem[{Sau}14]{Sa}
R.~{Sauer}.
\newblock {Volume and homology growth of aspherical manifolds}.
\newblock {\em ArXiv e-prints}, March 2014.

\bibitem[{Sch}13]{Scholze}
P.~{Scholze}.
\newblock {On torsion in the cohomology of locally symmetric varieties}.
\newblock {\em ArXiv e-prints}, June 2013.

\bibitem[Wal66]{Wa}
C.~T.~C. Wall.
\newblock Surgery of non-simply-connected manifolds.
\newblock {\em Ann. of Math. (2)}, 84:217--276, 1966.

\end{thebibliography}

\end{document}